\documentclass[11pt]{article}
\usepackage{amsmath, amssymb, eucal, latexsym, multicol, amsthm, bm}
\usepackage{psfrag}
\usepackage{comment}
\usepackage{comment,graphicx,pstcol}
\usepackage[figtopcap,FIGBOTCAP]{subfigure}

\usepackage{epstopdf, pstricks}

\usepackage{enumerate}

\input definitions.sty

\title{{Contraction and optimality properties\\ of an adaptive Legendre-Galerkin method:\\ the multi-dimensional case}}
\author{Claudio Canuto \thanks{Dipartimento di Scienze Matematiche,
Politecnico di Torino,
Corso Duca degli Abruzzi 24,
I-10129 Torino, Italy (claudio.canuto@polito.it )} \and Valeria Simoncini\thanks{Dipartimento di Matematica,
Universit\`a di Bologna,
Piazza di Porta San Donato  5, I-40127 Bologna, Italy
(valeria.simoncini@unibo.it).}\and Marco Verani\thanks{MOX-Dipartimento di Matematica, Politecnico di Milano, P.zza Leonardo Da Vinci 32, I-20133 Milano, Italy (marco.verani@polimi.it).}}

\date{\today}

\begin{document}

\maketitle
%%%%%%%%%%%%%%%%%%%%%%%%%%%%%%%%%%%%%%%%%%%%%%%%%%%%%%%%%%%%%%%%%%%%%%%%%%%%%%%

\pagestyle{myheadings}
\thispagestyle{plain}
\markboth{}
        {}

\begin{abstract}
We analyze the theoretical properties of an adaptive Legendre-Galerkin method
in the multidimensional case. After the recent investigations  for
Fourier-Galerkin methods in a periodic box and for Legendre-Galerkin methods in the one dimensional setting, the present study represents a further step 
towards a mathematically rigorous understanding of
adaptive spectral/$hp$ discretizations of elliptic boundary-value
problems. The main contribution of the paper is a careful construction of a multidimensional Riesz basis in $H^1$, based on a quasi-orthonormalization procedure.  This allows us to design an adaptive algorithm,  to prove its convergence by a contraction argument, and to discuss its optimality properties 
(in the sense of non-linear approximation theory) in certain sparsity classes of Gevrey type.
\end{abstract}

%\begin{AMS}
%\end{AMS}

%\begin{keywords}
%\end{keywords}

%\tableofcontents

%%%%%%%%%%%%%%%%%%%%%%%%%%%%%%%%%%%%%%%%%%%%%%%%%%%%%%%%%%%%%%%%%%%%%%%%%%%%%%%%%%%%%%
\section{Introduction}\label{S:intro}

The use of adaptivity in numerical modelling and simulation has now become a standard in Engineering and industrial applications.  Although the practice goes back to the $70$'s, the mathematical understanding of the convergence and optimality properties of adaptive algorithms for approximating the solution of multidimensional PDEs is rather recent. 
For linear elliptic problems the first convergence
results of adaptive $h$-type finite element methods ($h$-AFEM) have been proved by D\"orfler \cite{dorfler:96} and Morin, Nochetto, and Siebert \cite{MNS:00}. On  the other hand, the first convergence rates were derived for wavelets in any dimensions by Cohen, Dahmen, and DeVore \cite{CDDV:1998}, and for $h$-AFEM by Binev, Dahmen, and DeVore \cite{BDD:04} for the two-dimensional case and Stevenson \cite{Stevenson:2007} for any dimensions.  The most general results
for $h$-AFEM are those contained in Casc\'on, Kreuzer, Nochetto,
and Siebert \cite{Nochetto-et-al:2008} for any dimensions and $L^2$ data,
and in Cohen, DeVore, and Nochetto \cite{CDN:11} for two-dimensional case and $H^{-1}$ data. 
The key result of this theory is that  wavelets and $h$-AFEM are capable of guaranteeing convergence rates coherent with those dictated by 
the (best $N$-term) approximation classes where the solution and data 
belong. However, 
in the above wavelet and FEM contexts,  convergence rates are limited by the approximation power of
the method, which is finite and related to the
polynomial degree of the basis functions or the number of their vanishing moments, as well as the 
{\em sparsity}  of the solution and the data. The latter is always measured in an {\em algebraic} approximation class, i.e.,
the best $N$-term approximation error decays at least as a power of $N^{-1}$.
We refer to the surveys \cite{NSV:09} by Nochetto, Siebert and Veeser
for AFEM and \cite{Stevenson:09} by Stevenson for adaptive wavelets.

For adaptive methods with infinite approximation power (such as spectral or spectral-element methods, and $hp$-type finite element methods), the state of the art is less developed. Although the numerical implementation started long time ago and has led to the design of very sophisticated and efficient adaptive $hp$ algorithms (see, e.g., \cite{Mitchell:2011}; see also \cite{CV:book} and the references therein), very little is known on their theoretical properties. In particular, 
after the pioneering work \cite{Babuska-Gui} focussed on the approximation of specific types of functions,
some rigorous convergence results for the $hp$ adaptive solution of elliptic problems have been obtained only recently in \cite{Schmidt-Siebert:00,Doerfler-Heuveline:07,Burg-Doerfler:11}. However, these studies do not address any optimality analysis.

A first step in this direction has been accomplished in \cite{CNV:mathcomp} by considering adaptive spectral Fourier-Galerkin methods 
in a periodic box in $\mathbb{R}^d$, $d\ge1$, which represent the simplest instance of infinite-order methods yet providing a very important conceptual benchmark. The contraction and the optimal cardinality properties of various algorithms are presented therein;  in the analysis,  suitable nonlinear approximation classes
(also termed sparsity classes) are involved, namely the already mentioned algebraic classes and the newly introduced  {\em exponential} classes corresponding to a (sub-)exponential
decay of the best $N$-term approximation error. The latter classes, of Gevrey type, are natural to describe situations that  motivate the use of high-order methods.

A second step towards the study of optimality for high-order methods has been performed in \cite{Canuto-Nochetto-Verani-CMA:2014} where the method and the results contained in \cite{CNV:mathcomp} have been extended to a non-periodic setting in one dimension.   Such a setting is the closest to the periodic one, since an $H^1$-orthonormal basis is readily available (the so-called Babu\v ska-Shen basis formed by the primitives of the Legendre polynomials); together with the associated dual basis, it allows one to represent the norm of a function or a functional (e.g., the residual associated to the approximate solution) as an $\ell_2$-type norm of the vector of its expansion coefficients. Furthermore, the use of an orthonormal basis allows the efficient implementation of the greedy and coarsening procedures required by the adaptive algorithm. Indeed, the study of the optimality properties performed in  \cite{CNV:mathcomp,Canuto-Nochetto-Verani-CMA:2014} relies on a careful analysis of the relation between the sparsity class of a function and the sparsity class of its image through the differential operator. This analysis is based on the observation that the stiffness matrix associated  to a differential operator with smooth coefficients exhibits a quasi-sparse behavior, i.e., an exponential decay of its entries as one goes away from the diagonal. The discrepancy between the sparsity classes of the residual and the exact solution suggests the introduction of a coarsening step that guarantees the optimality of the computed approximation at the end of each adaptive iteration.
 
The present paper deals with adaptive Legendre-Galerkin methods in a tensorial domain in $\mathbb{R}^d$, $d >1$, for elliptic equations submitted to Dirichlet boundary conditions. This poses additional difficulties with respect to the one dimensional case, considered in \cite{Canuto-Nochetto-Verani-CMA:2014}. In particular, the crucial issue is represented by the $H^1$-stability properties  of the multidimensional Legendre polynomials. Unfortunately, the natural basis, formed by tensor products of one-dimensional basis functions, is not $H^1$-orthogonal
(because, unlike the Fourier basis, the one-dimensional Babu\v ska-Shen basis is not simultaneously orthogonal in $L^2$ and $H^1$) and not even a Riesz basis. This suggests searching for a Riesz basis in $H^1$, still remaining closely related to the tensorial BS basis in order to take advantage of the properties
of Legendre polynomials. The main idea developed in this paper is to start with a Gram-Schmidt (GS) orthogonalization of the latter basis, but then apply a controlled thresholding procedure that discards the  smallest contributions from the
linear combinations generated by GS: in other words, we devise a quasi-orthonormalization technique. A fundamental ingredient for rigorously controlling this procedure is the construction of sharp estimates on the decay of  the GS coefficients, which in turn involve the decay of the entries of the inverse stiffness matrix for the Laplacian \cite{CSV:13}. With such a new basis, results comparable to those of  
\cite{CNV:mathcomp,Canuto-Nochetto-Verani-CMA:2014} can be established. In particular, they rely on the exponential decay of the entries of the stiffness matrix, when the differential operators has smooth (analytic) coefficients,
and on the repeated application of a coarsening stage in the adaptive algorithm. 

The results of the present paper can be easily extended to cover the case of adaptive ``$p$-type" spectral element methods, i.e., when the domain is decomposed in a fixed number of (images of) tensorial elements, and adaptivity concerns the choice of the expansion functions in each element. Furthermore, some of the ideas and methods here introduced could influence the design of adaptive $hp$-type algorithms, as far as the phase of ``$p$-enrichment" within the elements
is concerned. With this respect, we remark that very recently, the optimality properties of an $hp$-adaptive finite element method have been obtained in \cite{CNSV:14}, employing the pioneering results on $hp$-tree approximation of \cite{B:oberw,B:tree}.

The outline of the paper is as follows. In Section \ref{sec:basis} we detail the construction of our multidimensional Riesz basis and provide the reader with both theoretical results and quantitative insight, the latter concerning in particular the compression properties of the thresholding procedure. In Section \ref{sec:gen} we introduce the algebraic representation of an elliptic differential problem in terms of the above Riesz basis and discuss the exponential decay properties of the entries of the corresponding stiffness matrix. Finally, in Section \ref{sec:adapt-alg} we  present our adaptive Legendre algorithm (FPC-ADLEG) and prove its contraction and optimality properties. 

\bigskip

Throughout the paper,  $A \lsim B$ means $A \le c \, B$ for some constant $c>0$ 
independent of the relevant parameters in the inequality; $A \simeq B$ means 
$B \lsim A \lsim B$.

%%%%%%%%%%%%%%%%%%%%%%%%%%%%%%%%%%%%%%%%%%%%%%%%%%%%%%%%%%%%%%%%%%%%%%%%%%%%%%%%%%%%%%

%%%%%%%%%%%%%%%%%%%%%%%%%%%%%%%%%%%%%%%%%%%%%%%
%------------------------------------------------------------------------------------------------------
\section{Modal bases in $H^1_0$ and norm representations}\label{sec:basis}
%------------------------------------------------------------------------------------------------------

We start with the one-dimensional case. Set $I=(-1,1)$ and let 
$L_k({x})$, $k \geq 0$, stand for the $k$-th Legendre orthogonal polynomial in $I$, 
which satisfies  ${\rm deg}\, L_k = k$, $L_k(1)=1$ and
\begin{equation}\label{eq:Leg-ort}
\int_I L_k({x}) L_m({x}) \, d{x} = \frac2{2k+1}\, \delta_{km}\;, \qquad m \geq 0 \;.
\end{equation}
The natural modal basis in $H^1_0(I)$ is the {\sl Babu\v ska-Shen basis} (BS basis), whose elements are defined as
\begin{equation}\label{eq:defBS}
\eta_k({x})=\sqrt{\frac{2k-1}2}\int_{{x}}^1 L_{k-1}(s)\,{d}s =
\frac1{\sqrt{4k-2}}\big(L_{k-2}({x})-L_{k}({x})\big)\ , 
 \qquad k \geq 2\;.
\end{equation}
The basis elements satisfy ${\rm deg}\, \eta_k = k$ and
\begin{equation}\label{eq:propBS.1}
(\eta_k,\eta_m)_{H^1_0({I})} = 
\int_I \eta_k'({x}) \eta_m'({x}) \, d{x} = \delta_{km}\;, \qquad k,m \geq 2 \;,
\end{equation}
i.e., they form an orthonormal basis for the ${H^1_0({I})}$-inner product. Equivalently, 
the (semi-infinite) stiffness matrix ${S}_\eta$ of the Babu\v ska-Shen basis with respect to this inner product is
the identity matrix ${I}$. On the other hand, one has
\begin{equation}\label{eq:propBS.2}
(\eta_k,\eta_m)_{L^2({I})} = 
\begin{cases}
\frac2{(2k-3)(2k+1)} & \text{if } m=k \;, \\
- \frac1{(2k+1)\sqrt{(2k-1)(2k+3)}} &  \text{if } m=k+2 \;, \\
0 & \text{elsewhere.}
\end{cases} \; \qquad \text{for } k \geq m \;, 
\end{equation}
which means that the mass matrix ${M}_\eta$ is pentadiagonal with only three non-zero entries per row. (Since even and odd modes are mutually orthogonal,
the mass matrix could be equivalently represented by a couple of tridiagonal matrices, each one collecting the inner
products of all modes with equal parity.)  

Any $v \in H^1_0({I})$ can be expanded in terms of the Babu\v ska-Shen basis, as
$v = \sum_{k =2}^\infty \hat{v}_k \eta_k$ with $\hat{v}_k = (v,\eta_k)_{H^1_0({I})}$
and its $H^1_0({I})$-norm can be expressed, according to the Parseval identity, as
\begin{equation}\label{eq:propBS.3}
\Vert v \Vert_{H^1_0({I})}^2  =  \sum_{k= 2}^\infty |\hat{v}_k|^2 = \hat{v}^T \hat{v}\;,
\end{equation}
where the vector $\hat{v}=(\hat{v}_k)$ collects the coefficients of $v$. The $L^2(I)$-norm of $v$ is
given by
\begin{equation}\label{eq:propBS.3bis}
\Vert v \Vert_{L^2({I})}^2  = \hat{v}^T  {M}_\eta \, \hat{v}\;.
\end{equation}
Correspondingly, any element $f \in H^{-1}({I})$ can be expanded along the {\sl dual Babu\v ska-Shen basis}, whose elements
$\eta_k^*$, $k \geq 2$, are defined by the conditions
$\langle \eta_k^* , v \rangle = \hat{v}_k$ $\forall v \in H^1_0({I})$,
precisely one has
$f = \sum_{k = 2}^\infty \hat{f}_k \eta_k^*$ with $\hat{f}_k = \langle f,\eta_k \rangle$,
and its $H^{-1}({I})$-norm can be expressed, according to the Parseval identity, as
\begin{equation}\label{eq:propBS.4}
\Vert f \Vert_{H^{-1}({I})}^2  =  \sum_{k= 2}^\infty |\hat{f}_k|^2 \;.
\end{equation}

Summarizing, we see that the one-dimensional Legendre case is perfectly similar, from the point of view of 
expansions and norm representations, to the Fourier case \cite{CNV:mathcomp}. The situation changes significantly in higher dimensions.
For the sake of simplicity, we confine ourselves to the case of dimension $d=2$, since higher dimensions pose
no conceptual difficulties but require a larger computational effort in the numerical experiments.

\medskip  
Let us set $\Omega=(-1,1)^2$ and let us consider in $H^1_0(\Omega)$ the {\sl tensorized Babu\v ska-Shen basis}, 
whose elements are defined as
\begin{equation}\label{eq:defBS.2}
\eta_k(x) = \eta_{k_1}(x_1) \eta_{k_2}(x_2)\;, \qquad k_1, k_2 \geq 2 \;, 
\end{equation}
where we set $k=(k_1,k_2)$ and $x=(x_1, x_2)$; indices vary in 
${\cal K}=\{k \in \mathbb{N}^2 \ : \ k_i \geq 2 \text{ for }i=1,2 \}$.

The tensorized BS basis is no longer orthogonal, since
\begin{equation}\label{eq:orthogonal}
(\eta_k,\eta_m)_{H^1_0(\Omega)} = (\eta_{k_1},\eta_{m_1})_{H^1_0({I})}(\eta_{k_2},\eta_{m_2})_{L^2({I})}+
 (\eta_{k_1},\eta_{m_1})_{L^2({I})}(\eta_{k_2},\eta_{m_2})_{H^1_0({I})} \;,
\end{equation}
hence, by \eqref{eq:propBS.1} and \eqref{eq:propBS.2}, we have $ (\eta_k,\eta_m)_{H^1_0(\Omega)} \not = 0$ if and only if $k_1=m_1$ and $k_2-m_2 \in \{-2,0,2\}$, or
$k_2=m_2$ and $k_1-m_1 \in \{-2,0,2\}$. 

In the sequel, this basis and the corresponding index set $\mathcal{K}$ will be ordered by increasing total
degree $k_\text{tot}=k_1+k_2$ and, for the same total degree, by increasing values of $k_1$ 
(this will be referred to as the ``A" ordering, see Fig.~\ref{Fig:ordinamenti-A}). We will use the
following notational convention:
$$
m < k \qquad \text{means that } \eta_m \text{ preceeds } \eta_k \text{ in this ordering.}
$$
With this ordering, let us denote again by ${S}_\eta$ the stiffness matrix of the
tensorized Babu\v ska-Shen basis with respect to the $H^1_0(\Omega)$-inner product. The matrix $S_\eta$ is infinite-dimensional. 

In the sequel, we will also need finite dimensional matrices 
defined as follows. For fixed $p\geq 2$ define the set 
\begin{equation}
\mathcal{K}^{p}:=\{ k \in \mathcal{K}:  k_1+ k_2 \leq p \},
\end{equation}
with cardinality $\text{card}(\mathcal{K}^{p})\simeq p^2$,
which identifies the basis functions of total degree not greater than $p$, ordered as above.
The corresponding stiffness matrix will be denoted by $S_\eta^{p}$, which is a truncated version of $S_\eta$ (upper-left section); its sparsity pattern is shown in Fig.~\ref{fig:patt.Q}. Let us also introduce the index set 
\begin{equation}
\mathsquare{{\cal K}^{p}}:=\{k \in \mathcal{K} \ : \ k_i\leq p \text{ for }i=1,2\}
\end{equation}
whose elements, on the contrary,  are ordered lexicographically, i.e., by increasing values of $k_2$ and, for the same $k_2$, by increasing values of $k_1$ 
(this will be referred to as the ``B" ordering, see Fig.~\ref{Fig:ordinamenti-B}).
The corresponding stiffness matrix will be denoted by $\mathsquare{{S}^{p}_{\eta}}$. Recalling \eqref{eq:orthogonal} it is the sum of two Kronecker  products, i.e., $\mathsquare{{S}^{p}_{\eta}}= M_{p}\otimes I_p + I_p \otimes M_{p}$, where $M_p$ is the one-dimensional mass matrix truncated at order $p$ and $I_p$ is the identity matrix of the same size.
We observe that the position of an index 
 $k\in \mathsquare{{\cal K}^{p}}$
does depend on $p$ because of the lexicographical ordering; on the contrary,
the position of $k\in {\cal K}^{p}$ is independent of $p$ because it coincides with the position in the infinite dimensional index set $\mathcal{K}$. We note for further reference that 
$$\mathsquare{{\cal K}^{\tilde{p}}} \subset {\cal K}^{p} \subset \mathsquare{{\cal K}^{p}},$$
where $\tilde{p}$ is the integer part of $p/2$. This implies, thanks to a Rayleigh quotient argument, that
\begin{equation}\label{eigenvalues}
\lambda_{\min}(\mathsquare{{S}_{\eta}^{\tilde{p}}})\geq \lambda_{\min}({S}_{\eta}^{p}) \geq \lambda_{\min} (\mathsquare{{S}_{\eta}^{p}})\quad\text{and}\quad
\lambda_{\max}(\mathsquare{{S}_{\eta}^{\tilde{p}}})\leq \lambda_{\max}({S}_{\eta}^{p}) \leq \lambda_{\max} (\mathsquare{{S}_{\eta}^{p}}).
\end{equation}

\begin{figure}[t!]
\begin{center}
\subfigure[``A'' ordering of ${\cal K}^{p}$]{\includegraphics[width=0.32\textwidth]{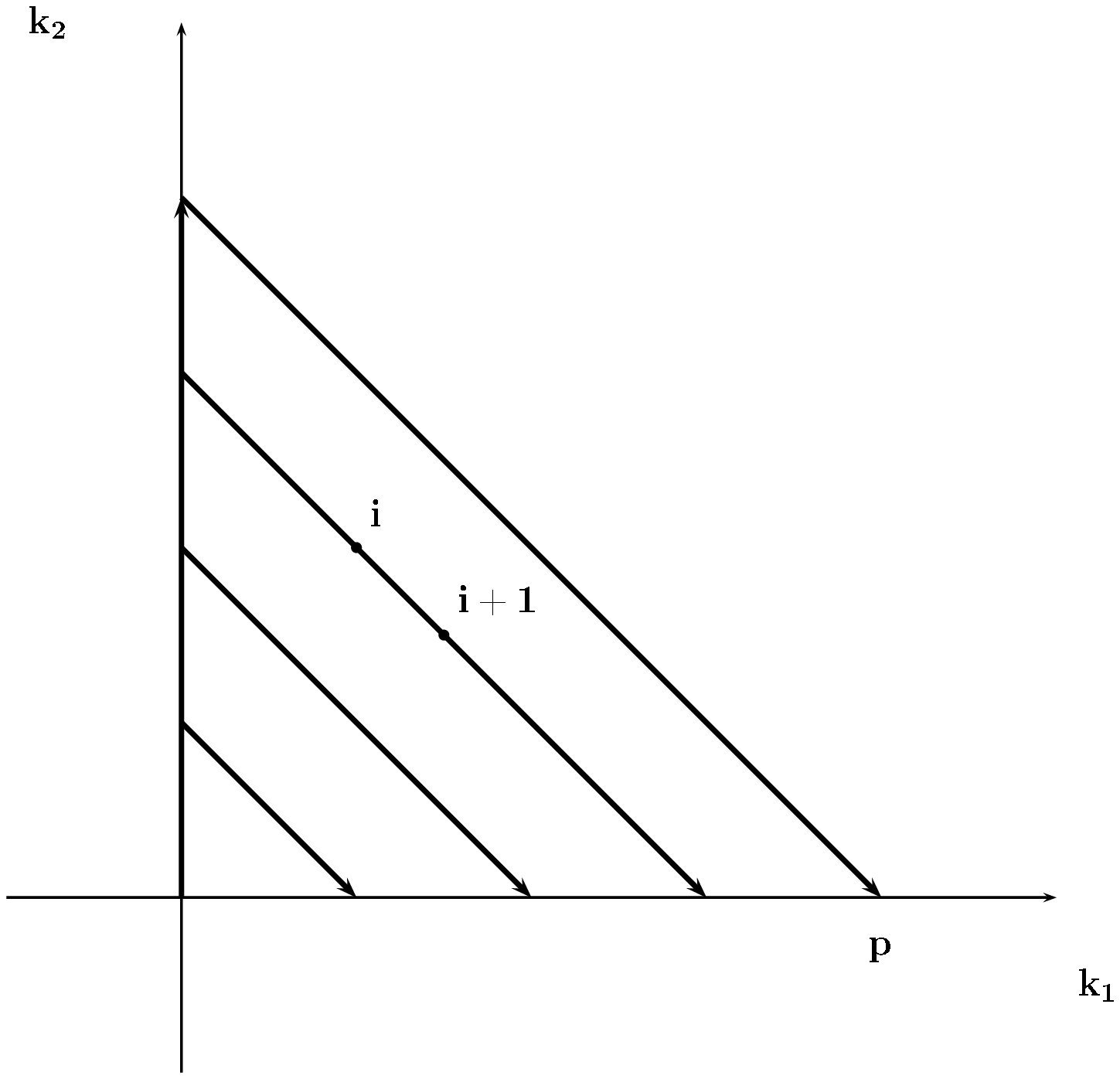}\label{Fig:ordinamenti-A}}
\subfigure[``B'' ordering of $\mathsquare{{\cal K}^{p}}$]{\includegraphics[width=0.32\textwidth]{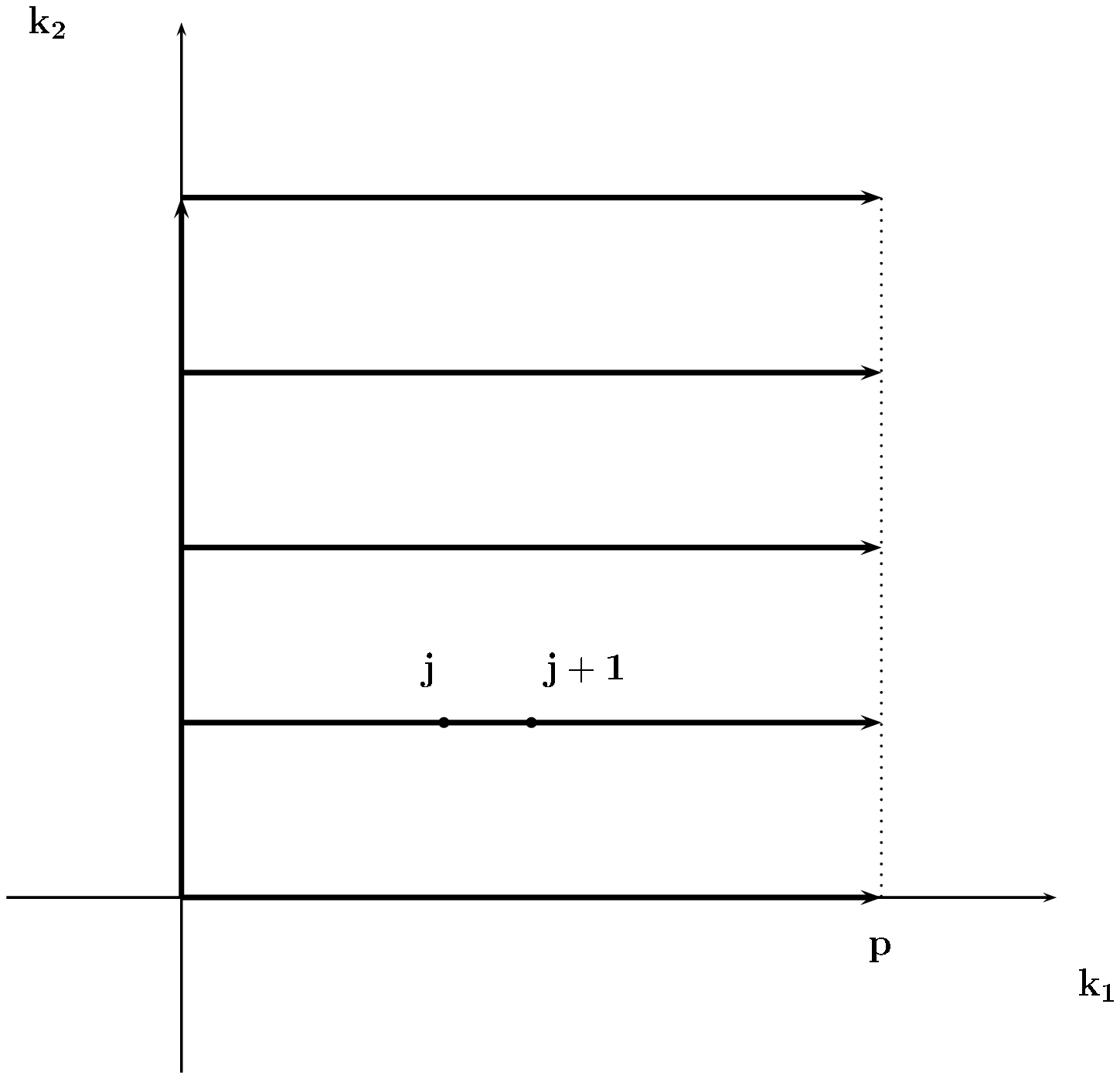}\label{Fig:ordinamenti-B}
}
\subfigure[``C'' ordering of $\mathsquare{{\cal K}^{p}}$]{\includegraphics[width=0.33\textwidth]{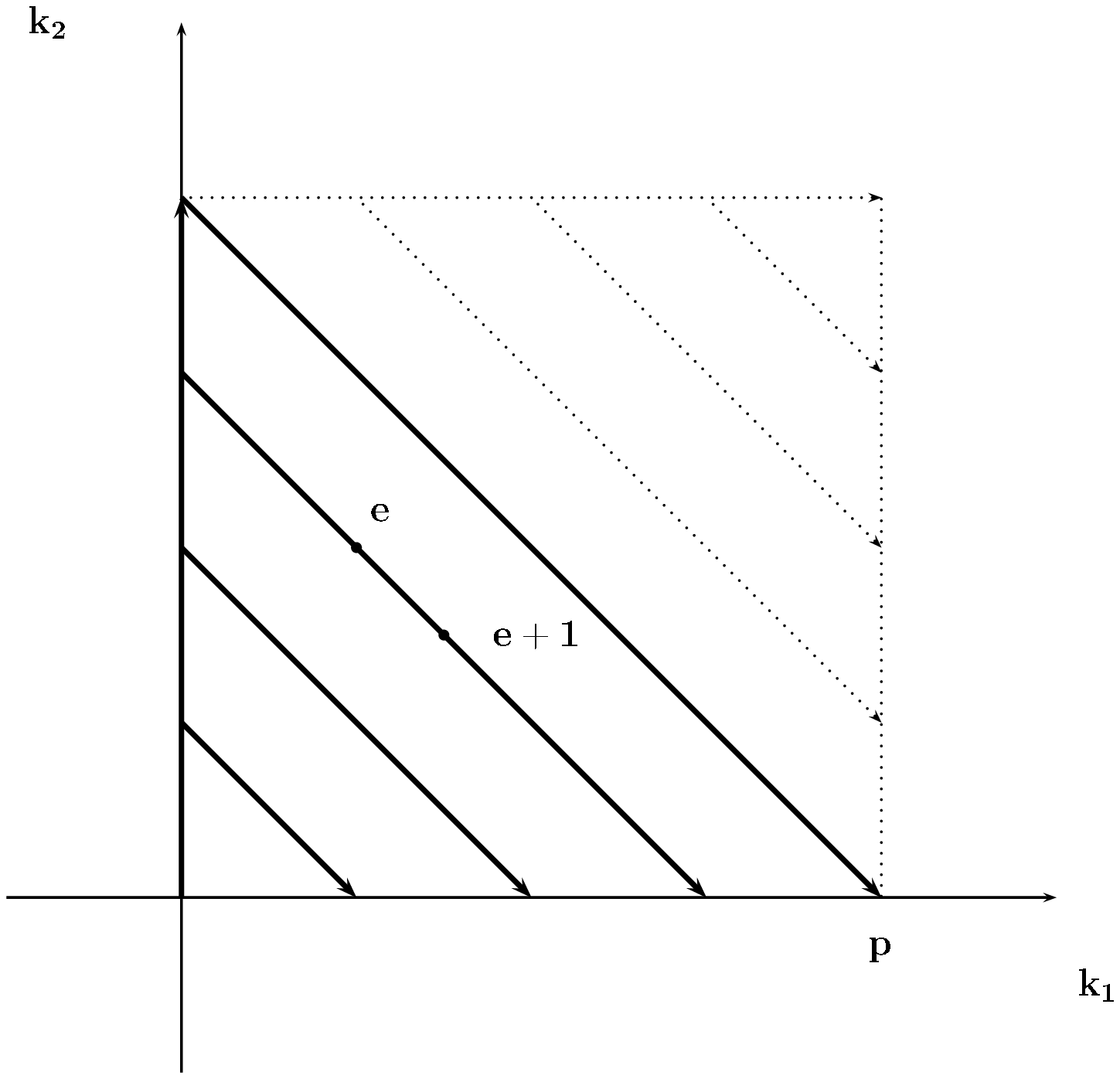}\label{Fig:ordinamenti-C}}
\end{center}
\caption{Orderings of index sets.}
\end{figure}

%\begin{figure}[t!]
%\begin{center}
%\includegraphics[width=0.48\textwidth]{fig1.pdf}
%\end{center}
%\caption{Pictorial representation of the C-ordering}\label{fig:1-R}
%\end{figure}

%------------------------------------------------------------------------------------------------------------------------
\subsection{Orthonormalization and quasi-orthonormalization}
%------------------------------------------------------------------------------------------------------------------------

Given any $v \in H^1_0(\Omega)$, let us expand it as $v = \sum_{k \in {\cal K}} \hat{v}_k \eta_k$ and let $\hat{v}$ be the
vector collecting its coefficients $\hat{v}_k$. Obviously, we cannot have a Parseval representation of the  $H^1_0(\Omega)$-norm of $v$ as in (\ref{eq:propBS.3}), since the basis is not orthonormal. However, we would be happy 
to have just
\begin{equation}\label{eq:propBS.5}
\Vert v \Vert_{H^1_0({\Omega})}^2  = \hat{v}^T {S}_\eta \, \hat{v} \simeq \hat{v}^T {\Delta}_\eta \, \hat{v}
= \sum_{k \in {\cal K}} |\hat{v}_k|^2 d_k\;,
\end{equation}
for some diagonal matrix ${\Delta}_\eta$. Taking as $\hat{v}$ each vector of the canonical basis, this
should imply 
$$
s_{kk} \simeq d_{k} \qquad \forall {k \in {\cal K}}\,,
$$
i.e., ${\Delta}_\eta$ should be uniformly spectrally equivalent to ${D}_{\eta}:={\sf diag}\,{S}_\eta$. Unfortunately,
the eigenvalues of the generalized eigenvalue problem 
${S}_\eta \, {w} = \lambda \, {D}_\eta\, {w}$
are not uniformly bounded away from $0$ and $+\infty$.  
These eigenvalues are indeed the eigenvalues of the matrix 
$\widetilde{S}_{\eta}= {D}_{\eta}^{-1/2} {S}_{\eta} {D}_{\eta}^{-1/2}$ which is the stiffness matrix of the $H_0^1$-normalized BS basis. In particular, if we consider the finite dimensional matrices 
$\mathsquare{{S}_{\eta}^{p}}$, it is known \cite[Proposition 5]{Maitre:1996} that their largest eigenvalues are uniformly bounded, but the smallest eigenvalues tend to $0$ as $p^{-2}$. Due to \eqref{eigenvalues} the same results hold for the matrices  $S_{\eta}^{p}$. 
In conclusion, there is no hope to have (\ref{eq:propBS.5}) with the tensorized BS basis, and 
a new basis has to be sought.

\subsubsection{Orthonormalization}
We have pursued the idea of orthonormalizing the BS basis, since, as shown above, many inner-products between its functions are indeed zero. To this end, we first observe that the BS basis functions can be grouped in four families depending on their parity in each of the two variables (recall that the univariate BS basis functions are alternately even and odd); according to \eqref{eq:orthogonal} functions belonging to different families are mutually $H_0^1$-orthogonal. Consequently, after reordering of rows and columns, the matrix ${S}_{\eta}$ is block-diagonal with four blocks ${S}^{++}_{\eta}$,
${S}^{+-}_{\eta}$, ${S}^{-+}_{\eta}$ and ${S}^{--}_{\eta}$ corresponding to all combinations of even ($+$) or odd ($-$) one-dimensional basis functions in each direction. In addition, it is convenient to deal with $H^1_0$-normalized basis functions, which lead to blocks $\widetilde{{S}}^{\pm \pm}_{\eta}$. 
For notational simplicity, from now on any normalized block $\widetilde{{S}}^{\pm\pm}_{\eta}$ will be again denoted by ${S}_{\eta}$ and the corresponding basis functions will still be indicated by $\eta_k$ for $k$ belonging to an index set again denoted by $\cal{K}$. Since the four blocks behave in an equivalent way, in the following numerical results will be given only for the even-even case.
%We observe that ${S}_{\eta}$ has a pattern similar to the one corresponding to fi its difference approximation of the laplacian operator since the non-zero off-diagonal

As a first step we resort to the modified Gram-Schmidt
algorithm (see e.g. \cite{Golub-VanLoan:1996}), which allows one to build a sequence of functions
\begin{equation}\label{eq:defOBS}
\Phi_k = \sum_{m \leq k}  g_{mk} \eta_m \;,
\end{equation}
such that  $g_{kk}\not = 0$ and 
$$
(\Phi_k,\Phi_m)_{H^1_0(\Omega)} = \delta_{km} \qquad \forall \ k, m \in {\cal K}\;.
$$
We will refer to the collection $\Phi:=\{\Phi_k:\ k \in\mathcal{K}\}$ as the {\sl orthonormal Babu\v ska-Shen basis} (OBS basis) of the above chosen parity; obviously,
the associated stiffness matrix ${S}_\Phi$ with respect to the $H^1_0(\Omega)$-inner product is the identity matrix.
Equivalently, if ${G}=(g_{mk})$ is the upper triangular matrix which collects the coefficients generated by the
modified Gram-Schmidt algorithm above, one has
\begin{equation}\label{eq:propOBS.1}
{G}^T {S}_\eta {G} = {S}_\Phi = {I} \;,
\end{equation}
which shows that ${L}:={G}^{-T}$ is the lower-triangular Cholesky factor of ${S}_\eta$. 
It is important to notice that for any finite dimensional section $S_{\eta}^{p} $ of $S_{\eta}$ a similar relation holds, namely
$$ \big({G^{p}}\big)^T {S}^{p}_\eta {G}^{p} = {S}_{\Phi^{p}} = {I}^{p} \;,$$
where $G^{p}$ is the upper-left section of the infinite dimensional matrix $G$ with the same size as $S_{\eta}^{p}$ and $ \Phi^{p}:=\{\Phi_k:\ k \in\mathcal{K}^{p}\}$.  
This is an obvious consequence of the structure of the Gram-Schmidt algorithm  and the fact that the ordering of the basis functions in $\mathcal{K}^{p}$ is the same as in $\mathcal{K}$. On the contrary, 
due to the different orderings of the basis functions in $\mathsquare{{\cal K}^{p}}$ and $\mathcal{K}$, the GS algorithm applied to the matrix $\mathsquare{{S}^{p}_\eta}$ gives rise to a matrix $\mathsquare{{G}^{p}}$
which cannot be obtained by simply truncating the infinite dimensional $G$.

Unfortunately, unlike ${S}_\eta$, which is very sparse, the upper triangular matrix ${G}$ is full. 
However, the elements of $G$ exhibit nice decay features which are exemplified in Figure  \ref{fig:patt.G} again for $p=60$. The intensity of grey indicates that the entries of $G$ decay to zero moving away from the main diagonal, with different rates depending on the column. Indeed, recalling formula \eqref{eq:defOBS}, it is meaningful to monitor the decay of the elements $g_{mk}$ of $G$ that belong to a given column $k$ for the row index $m$ decreasing  from $k$ to $1$.
Figures \ref{fig:slow-fast}(left) and \ref{fig:slow-fast}(right) are representative of two extreme behaviors. Each plot in the figures represents the elements of a column of $G$ starting from the the main diagonal and moving towards the first row. Figure \ref{fig:slow-fast}(left) refers to columns $98,221,338$ which exhibit a ``slow'' decay. This behavior is typical of those columns associated to an index $k\in\mathcal{K}$ with $k_1$ close to $k_2$. On the other hand, Figure \ref{fig:slow-fast}(right) refers to columns $105,231,351$ which exhibit a ``fast'' decay, a typical behavior of those columns associated to an index $k\in\mathcal{K}$ for which $k_1$ and $k_2$ are very different from each other.

%%%%%%%%%%%%%%%%%
\begin{figure}[t!]
\begin{center}
\subfigure[Sparsity pattern of $S_{\eta}^{60}$.]{\includegraphics[width=0.47\textwidth]{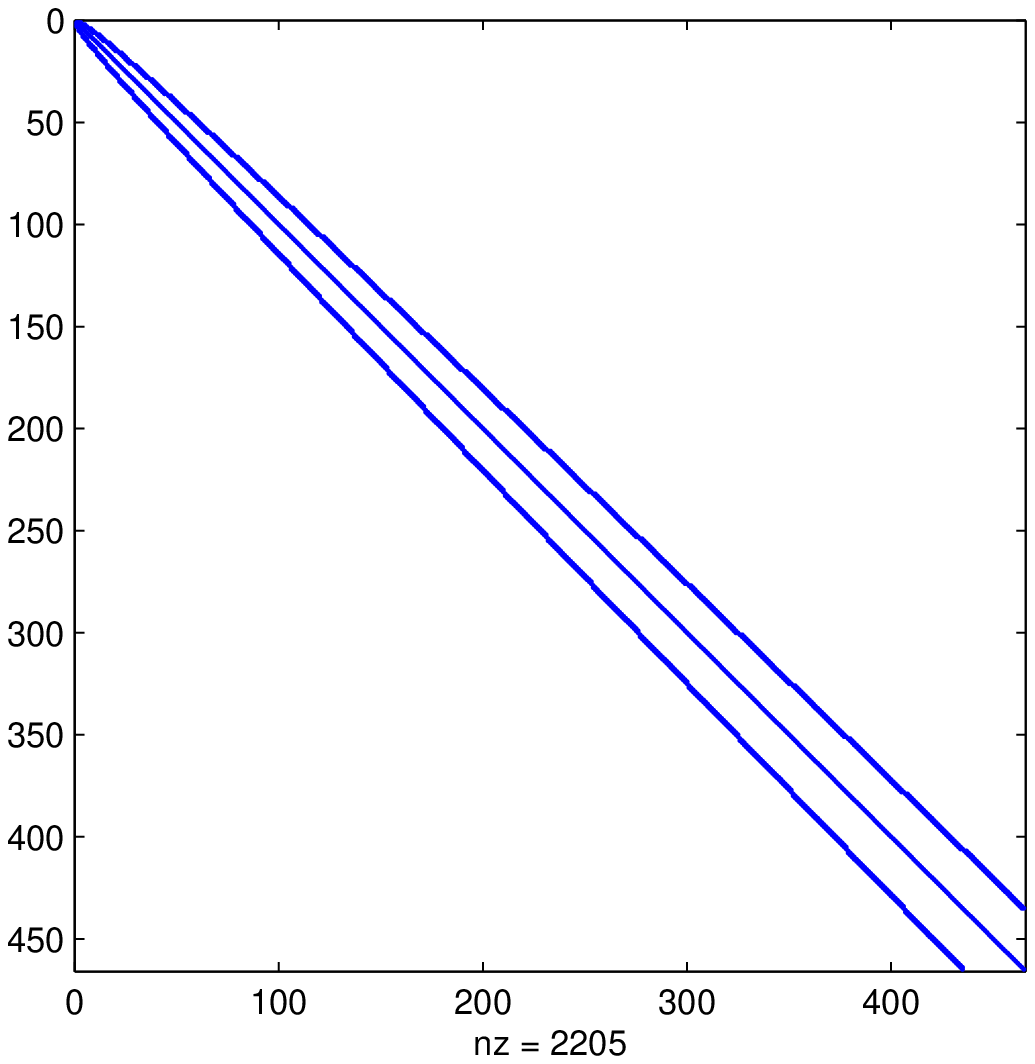}
\label{fig:patt.Q}}
\subfigure[Grey-scale size of the elements of $G^{60}$.]{\includegraphics[width=0.32\textwidth]{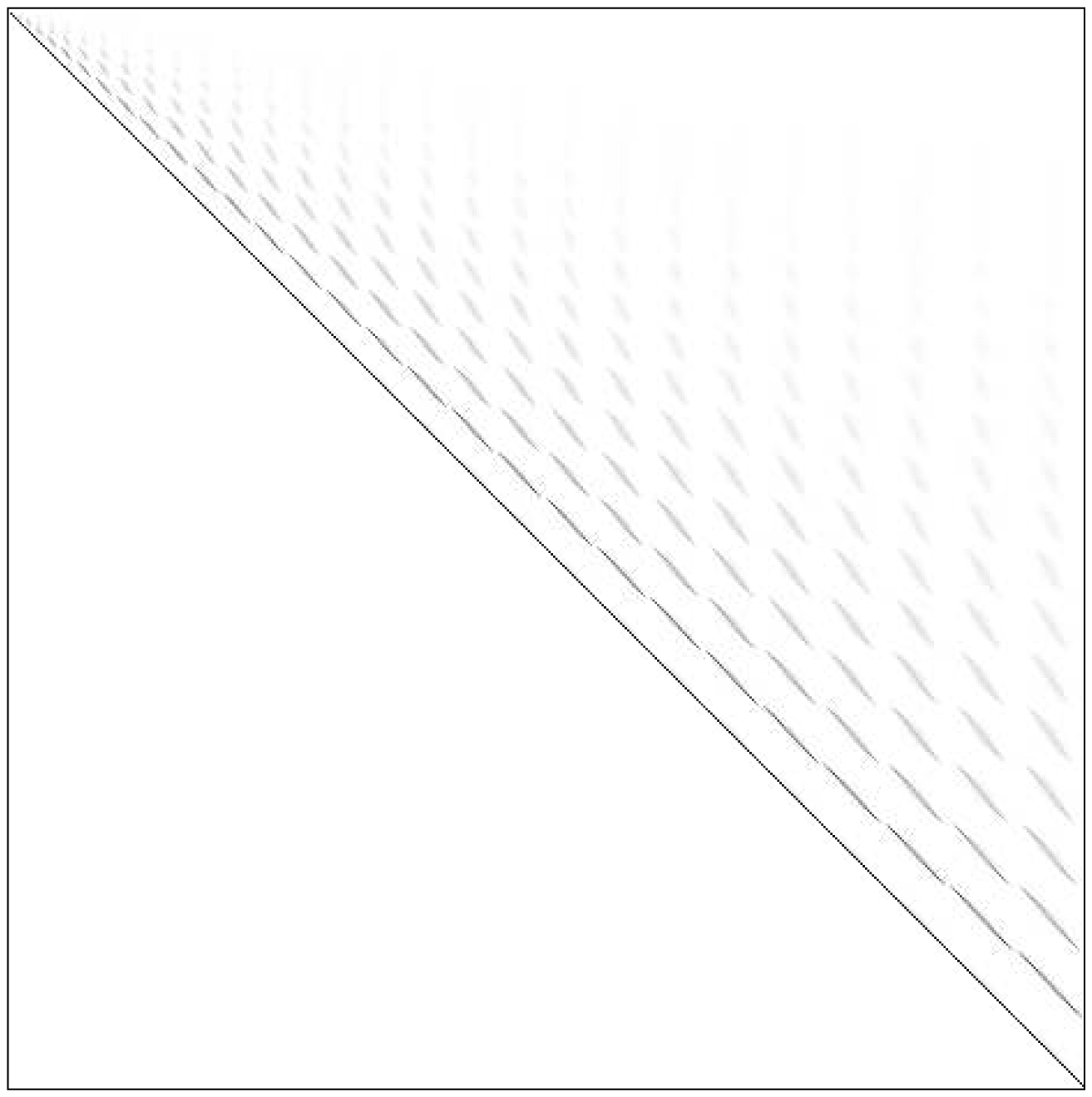}\label{fig:patt.G}}
\end{center}
\caption{Sparsity patterns}
\end{figure}
 
%%%%%%%%%%%
%\begin{figure}[t!]
%\begin{center}
%\includegraphics[width=0.65\textwidth]{G60-gray.eps}
%\end{center}
%\caption{Grey-scale size of the elements of $G^{60}$.}\label{fig:patt.G}
%\end{figure}

%%%%%%%%%
\begin{figure}[t!]
\begin{center}
\includegraphics[width=0.48\textwidth]{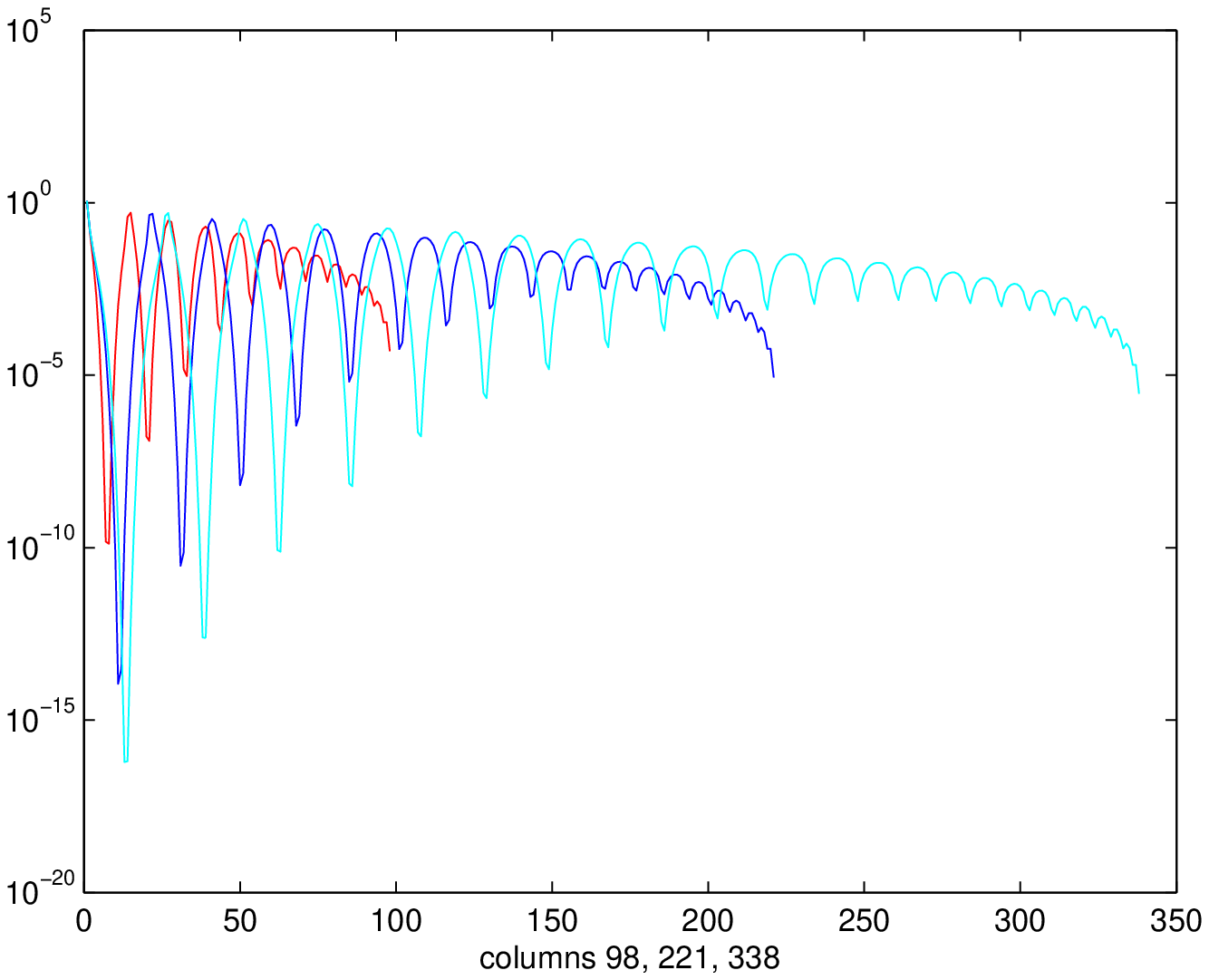}
\includegraphics[width=0.48\textwidth]{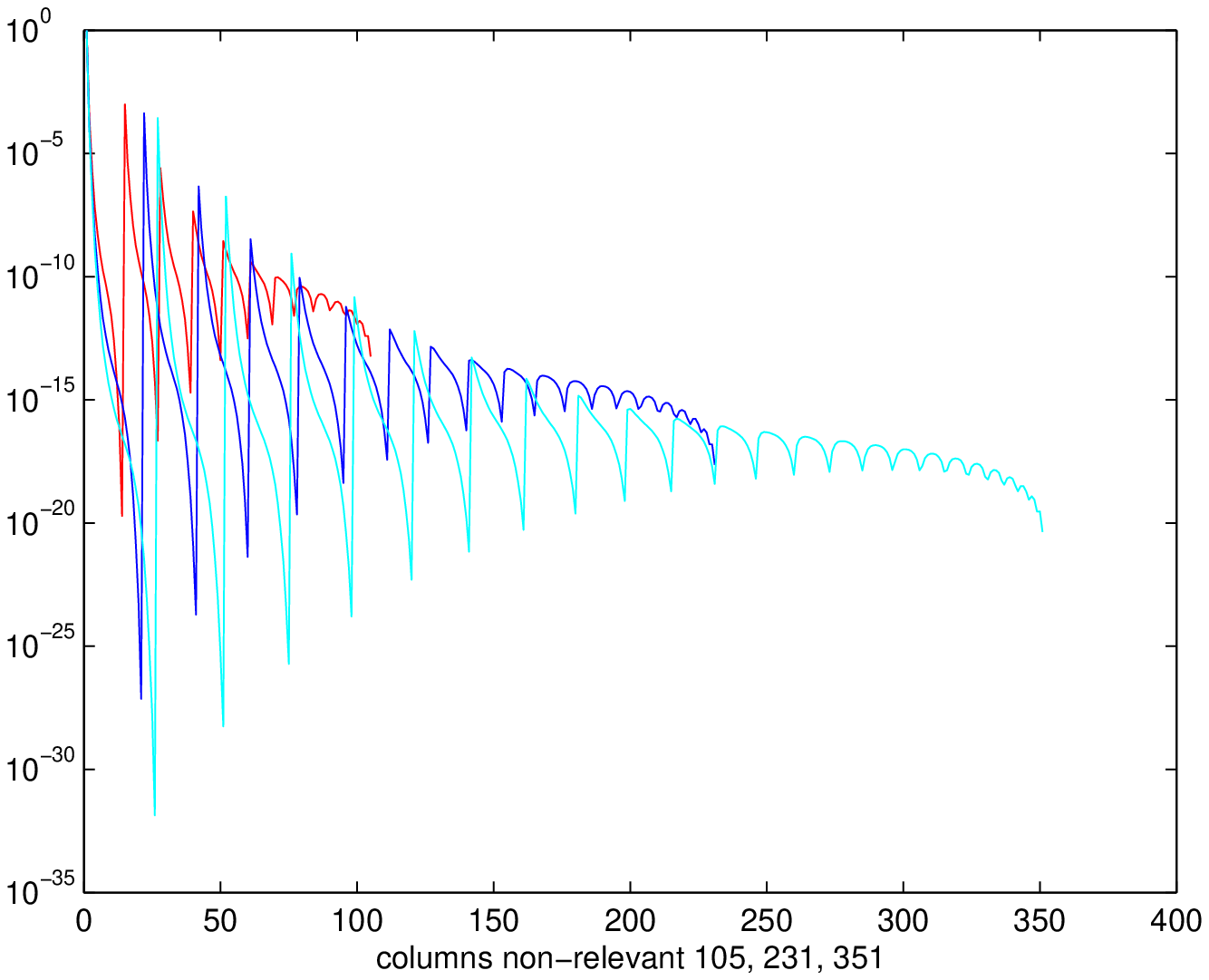}
\end{center}
\caption{Semi-logaritmic plot of some ``slow'' decaying columns (left) and ``fast'' decaying columns (right) of $G^{60}$ .}\label{fig:slow-fast}
\end{figure}

%%%%%%%%%

A theoretical upper bound for the elements of $G$ can be obtained applying Theorem 4.1 in \cite{Benzi-Tuma:2000}. 
It ensures that if $A$ is an SPD banded matrix with bandwidth $b$
such that $\max_i a_{ii}=1$, and if $A=LL^T$ is its
Cholesky factorization, then the entries of $G=L^{-T}$ obey an exponentially decaying 
bound away from the main diagonal, precisely 
\begin{equation}\label{eq:benzi-tuma}
|g_{ij}| \le \frac{2}{\lambda_{\min}}\rho^{j-i}, \qquad \rho = \left(\frac{\sqrt{\kappa}-1}{\sqrt{\kappa}+1}\right)^{2/b} ,
\end{equation}
where $\kappa=\lambda_{\max}/\lambda_{\min}$ is the condition number of $A$.
%In order to apply the theorem, let us consider an element $g_{mk}$ belonging to our matrix $G$.
%Assume that the multi-index $m=(m_1,m_2)$ corresponds to a single-index $i$ in the chosen ordering in 
%$\mathcal{K}$; similarly let the multi-index $k=(k_1,k_2)$ be associated to the single-index $j$.
%Let $p$ the smallest integer such that the index $k$ belongs to the set $\mathcal{K}^{p}$. 
%We set $p=p_j$ and note that $j\simeq p_j^2$. Then we apply estimate \eqref{eq:benzi-tuma}
%to the matrix $A=S_{\eta}^{[p_j]}$. As already noted above, we have 
%$$\kappa(S^{[p_j]}_{\eta})\simeq p_j^2\simeq j $$
%with $\lambda_{\min}\simeq p_j^{-2}\simeq j^{-1}$. Furthermore, if the element $(S_{\eta}^{[p_j]})_{\ell h}$
%is nonzero, then recalling \eqref{eq:orthogonal} it follows the total degrees $h_1+h_2$ and $\ell^{1}+\ell_2$
%differ by at most $2$ which means that $h$ and $\ell$ belong to two successive diagonals (see Figure \ref{Fig:ordinamenti}). 
%Then the  corresponding single-indices differ by at most $O(p_j)$ which implies that the bandwidth of $S_\eta^{[p_j]}$satisfies $b\simeq p_j \simeq \sqrt{j}$.
%As a consequence, 
%\begin{equation}
%\log \rho \simeq \frac{1}{\sqrt{j}}\log(1-\frac{C}{\sqrt{j}})\simeq \frac{1}{j}.
%\end{equation}
%We conclude that \eqref{eq:benzi-tuma} yields 
%\begin{equation}
%\vert g_{ij}\vert \leq C_1 j {\rm e}^{-\frac{C_2}{j} (j-i)}.
%\end{equation}
Note however that the observed decay of $g_{ij}$ is far from being monotonic (see Figures \ref{fig:slow-fast}(left) and \ref{fig:slow-fast}(right)). 

This oscillatory behavior can indeed be explained by resorting to the recent results presented in \cite{CSV:13}, as we are now going to detail. In what follows, for the ease of presentation, let ``A'', ``B'', and ``C'', resp., refer to the index orderings depicted in Fig. \ref{Fig:ordinamenti-A}, \ref{Fig:ordinamenti-B}, and \ref{Fig:ordinamenti-C}, resp.; the latter is a different ordering of the set $\mathsquare{{\cal K}^{p}}$, which coincides with the A-ordering on the subset ${\cal K}^{p}\subset\mathsquare{{\cal K}^{p}}$.
%
%\begin{figure}[htbp]
%  \centering 
% 
%   \psset{xunit=1cm} \psset{yunit=1cm}
%        \begin{pspicture}(-2,-2)(5,5)
%           
%     \psline[linewidth=1pt]{->}(-1,0)(5,0)
%          \psline[linewidth=1pt]{->}(0,-1)(0,5)
%         
%          \psline[linewidth=2pt]{->}(0,1)(1,0)
%           \psline[linewidth=2pt]{->}(0,2)(2,0)
%            \psline[linewidth=2pt]{->}(0,3)(3,0)
%             \psline[linewidth=2pt]{->}(0,4)(4,0)
%
%
%             \psline[linewidth=2pt]{->}(0,0)(0,4)
%
%               \uput*[0](3.4,-0.5){\Large${\bf p}$}
%               
%                 \uput*[0](-1,5){\Large${\bf k_2}$}
%                 \uput*[0](5,-0.5){\Large${\bf k_1}$}
%
%                
%                 \psdot[dotstyle=*,dotsize=4pt](1.5,1.5)
%                   \psdot[dotstyle=*,dotsize=4pt](1,2)
%               \uput[0](1.5,1.5){$e+1$}
%                 \uput[0](1,2.2){$e$}
%                 
%                 \psline[linewidth=1pt,linestyle=dotted]{-}(4,0)(4,4)
%                  \psline[linewidth=1pt,linestyle=dotted]{->}(0,4)(4,4)
%
%          \psline[linewidth=1pt, ,linestyle=dotted]{->}(1,4)(4,1)
%           \psline[linewidth=1pt,linestyle=dotted]{->}(2,4)(4,2)
%            \psline[linewidth=1pt,linestyle=dotted]{->}(3,4)(4,3)                 
%                          
%         
%     \end{pspicture}
%
%\caption{Pictorial representation of the C-ordering}\label{fig:1-R}
%\end{figure}
Furthermore, we drop the dependence on $p$ when using matrices. As the scheme C is an expansion of the scheme A,
the associated stiffness matrix $\SC$ has the form 
$$
\SC= 
\begin{bmatrix}
\SA & S_{12} \\ S_{12}^T & S_{22}
\end{bmatrix}
$$
where $\SA$ is the stiffness matrix associated with the scheme A.

Let $\SA = \LA \LA^T$ and let $\tLA \tLA^T =
\tilde S_{22} := S_{22} - S_{12}^T {{S}_{\eta}}^{-1}S_{12}$. It holds that
$$
\SC=
\begin{bmatrix}
\LA & 0 \\ 
S_{12}^T\LA^{-T} & \tLA
\end{bmatrix}
\begin{bmatrix}
\LA^T &  \LA^{-1} S_{12} \\ 0 & \tLA^T
\end{bmatrix} =: \LC \LC^T ,
$$
with $\LC$ banded with bandwidth $m$. Since all elements of $\SC$ are less
than one, we have 
\begin{equation}\label{aux:L} 
|{(\LC)}_{ij}|\le 1.
\end{equation}
Taking the inverse of the previous factorization, we have
$$
{\SC}^{-1} = \LC^{-T} \LC^{-1}=: \GC \GC^T
$$
with 
$$
\GC =
\begin{bmatrix}
\LA^{-T} &  -\LA^{-T} \LA^{-1} S_{12}  \tLA^{-T} \\ 
0 & \tLA^{-T}
\end{bmatrix} .
$$
%$$
%\begin{bmatrix}
%(\LA)^{-T} &  -(\LA)^{-T} (\LA)^{-1} S_{12}  (\tLA)^{-T} \\ 
%0 & (\tLA)^{-T}
%\end{bmatrix} 
%\begin{bmatrix}
%(\LA)^{-1} & 0 \\
%-(\LA)^{-1} S_{12}^T(\LA)^{-T}(\tLA)^{-1} & (\tLA)^{-1} 
%\end{bmatrix} 
%$$
This relation tells us that the elements of $\GA:=\LA^{-T}$ we are interested in, can be
read off from the upper left block of the factor $\GC$ of ${\SC}^{-1}$.

Now, we recall that
$$
\SC = P \SB P^T , \qquad
{\SC}^{-1} = P {\ISB}P^T ,
$$
where $\SB$ is the stiffness matrix obtained with a lexicographic
order (scheme B), and $P$ is a permutation matrix. It is worth observing that $\SC$ corresponds to the reverse Cuthill-McKee reordering of $\SB$ (see, e.g., \cite{Duff-Erisman-Reid:89}).

Denoting by $\pi(u)$  the permutation of the index $u$ defined by $P$, that is
$e_u^T P = e_{\pi(u)}^T $, we have 

\begin{equation}\label{eq:relation_inv}
({\SC}^{-1})_{uv} = ({\ISB})_{\pi(u),\pi(v)}.
\end{equation}

For every index $u$ in the diagonal ordering ``C'', $\pi(u)$ is the associated index in the lexicographical ordering ``B''. Next lemma details the construction of the map $\pi$.
\begin{lemma}
Let $n:=p-1$, $d$ with $1\leq d \leq 2n-1$ and $e$ with  $1\leq e \leq \min(d,n)-\max(0,d-n)$.
For 
\begin{equation}\label{eq:relation}
u=\sum_{s=1}^{d-1} (\min(s,n)-\max(0,s-n))+e,
\end{equation}
it holds
\begin{equation}\label{eq:map}
\pi(u)=n(\min(d,n)-e)+\max(0,d-n)+e.
\end{equation}
\end{lemma}
\begin{proof}
The parameter $d$ with $1\leq d \leq 2n-1$ is the index numbering the diagonals in the diagonal ordering ``C''. The index $d=1$ corresponds to the first diagonal made of a single element  (lower-left corner of the square in Figure \ref{Fig:ordinamenti-C}), while $d=n$ is associated to the main diagonal
and $d=2n-1$ to the last diagonal  (upper-right corner of the square in Figure \ref{Fig:ordinamenti-C}).
The parameter $e$ with $1\leq e \leq \min(d,n)-\max(0,d-n)$ is the index numbering the elements on the $d$-th diagonal (from upper-left to lower-right). 
Note that on the $d$-th diagonal there are exactly $\min(d,n)-\max(0,d-n)$ elements.
Let $(e,d)=(e(u), d(u))$ be such that $u=\sum_{s=1}^{d-1} (\min(s,n)-\max(0,s-n))+e$,
i.e., the element $u$ is associated to the $e$-th element on the $d$-th diagonal.
Then, straightforward calculations show that $\pi(u)=n(\min(d,n)-e)+\max(0,d-n)+e$.
\end{proof}

Let $u,v$ be two indices in the diagonal ordering ``C'' and $(e,d)=(e(u),d(u))$ and $(f,g)=(f(v),g(v))$ be 
such that equation \eqref{eq:relation} holds.  We preliminary want to estimate \eqref{eq:relation_inv}.
For ${\mathsquare{{S}_{\eta}^{p}}}$ we know that the following result holds.
\begin{proposition}\label{prop:decay1}(\cite[Proposition 2.7]{CSV:13}).
Let $\alpha=n(m-1)+\ell$ and $\beta=n(j-1)+i$
for proper choices of $i,j,\ell,m$. 
\begin{enumerate}
\item If $\ell=i$ or $m=j$ then there exists a positive constant $\gamma_1=\gamma_1(\kappa(\mathsquare{{S}_{\eta}}))$ such that 
\begin{equation}\label{eq:est_inv_B:1}
|(\ISB)_{\alpha,\beta}| \le \gamma_1 \frac 1 {\sqrt{{\mathfrak n}_1}}
\end{equation}
with ${\mathfrak n}_1 = {\mathfrak n}_1(\alpha,\beta)=\vert \ell-i\vert + \vert m-j\vert -1$.

\item If $\ell\not=i$ and $m\not=j$ then there exists a positive constant $\gamma_2=\gamma_2(\kappa(\mathsquare{{S}_{\eta}}))$
such that 
\begin{equation}\label{eq:est_inv_B:2}
|(\ISB)_{\alpha,\beta}| \le \gamma_2 \frac 1 {\sqrt{{\mathfrak n}_2}}
\end{equation}
with ${\mathfrak n}_2 = {\mathfrak n}_2(\alpha,\beta)=\vert \ell-i\vert + \vert m-j\vert -2$.
\end{enumerate}
\end{proposition}
Using \eqref{eq:map} and Proposition \ref{prop:decay1} we have the following result on the entry decay of the matrix ${(\SC)}^{-1}$. 
%%%%%%%%%%%%
\begin{corollary}\label{prop:decay-2}
Let $u,v$ be indexes in the C-ordering such that the following holds
\begin{eqnarray}
u=\sum_{s=1}^{d-1} (\min(s,n)-\max(0,s-n))+e,\,&&
v=\sum_{s=1}^{g-1} (\min(s,n)-\max(0,s-n))+f,\nonumber
\end{eqnarray}
for proper choices of the pairs $(d,e)$ and $(g,f)$. Consider \eqref{eq:est_inv_B:2}
for $\pi(u),\pi(v)$ such that 
\begin{eqnarray}
\pi(u)=n(\min(d,n)-e)+\max(0,d-n)+e,\,&& \pi(v)=n(\min(g,n)-f)+\max(0,g-n)+f.\nonumber
\end{eqnarray}
Let $i,j,\ell,m$ such that  $\pi(u)=n(m-1)+\ell$ and $\pi(v)=n(j-1)+i$ then it holds
%\begin{eqnarray}
%&&\pi(u)=n(m-1)+\ell \label{piu}\\ 
%&&\pi(v)=n(j-1)+i.\label{piv} 
%\end{eqnarray}
\begin{enumerate}
\item If $\ell=i$ or $m=j$ there exists a positive constant 
$\gamma_1=\gamma_1(\kappa(\SB))$ 
such that 
\begin{equation}\label{eq:est_inv_B:a}
|(\ISB)_{\pi(u),\pi(v)}| \le \gamma_1 \frac 1 {\sqrt{{\mathfrak n}_1}}
\end{equation}
with ${\mathfrak n}_1 = {\mathfrak n}_1(\pi(u),\pi(v))=\vert \ell-i\vert + \vert m-j\vert -1$.
\item If $\ell\not=i$ and $m\not=j$ then there exists a positive constant $\gamma_2=\gamma_2(\kappa(\SB))$ such that 
\begin{equation}\label{eq:est_inv_B:b}
|(\ISB)_{\pi(u),\pi(v)}| \le \gamma_2 \frac 1 {\sqrt{{\mathfrak n}_2}}
\end{equation}
with ${\mathfrak n}_2 = {\mathfrak n}_2(\pi(u),\pi(v))=\vert \ell-i\vert + \vert m-j\vert -2$.
\end{enumerate}
\end{corollary}
\begin{proof}
It is sufficient to employ Proposition \ref{prop:decay1} with $\alpha=\pi(u)$ and $\beta=\pi(v)$.
\end{proof}

Now we are ready to estimate the entries of $\GC={\SC}^{-1} \LC$.
Using Corollary \ref{prop:decay-2} we have the following estimate.
\begin{proposition}\label{prop:G-decay}
Under the assumptions of Corollary \ref{prop:decay-2}, for ${\mathfrak n}(\cdot,\cdot)={\mathfrak n}_i(\cdot,\cdot)$, $i=1,2$
depending on the values of the indexes $u,v$, it holds
\begin{equation}
|{(\GC)}_{u,v}| \le \gamma \frac{\tilde{p}} {{\sqrt{\mathfrak n(\pi(u),\pi(v^*))}} }\label{eq:inv_C} 
\end{equation}
where $\tilde{p}$ is the integer part of $p/2$ and $v^*=\text{argmin}_{v\leq w \leq v+\tilde{p}-1}  \mathfrak n(\pi(u),\pi(w))$.
\end{proposition}
\begin{proof}
It is enough to proceed as in \cite[Theorem 4.1] {Benzi-Tuma:2000}, with $\tilde{p}$ corresponding to the matrix bandwidth, to get the result. Indeed, using $\GC={(\SC)}^{-1} \LC$ together with \eqref{eq:relation_inv}, \eqref{eq:est_inv_B:a} and \eqref{eq:est_inv_B:b} it holds
\begin{eqnarray}
|{(\GC)}_{u,v}| &\le& \sum_{w=v}^{v+b-1} | ({\SC}^{-1})_{u,w}| \, |{(\LC)}_{w,v}|
= \sum_{w=v}^{v+b-1}  | ({\ISB})_{\pi(u),\pi(w)}|  |{(\LC)}_{w,v}| \nonumber\\
&\le& \gamma \frac b {{\sqrt{\mathfrak n(\pi(u),\pi(v^*))}} }
\end{eqnarray}
where we employ \eqref{aux:L} and set $v^*=\text{argmin}_{v\leq w \leq v+b-1}  \mathfrak n(\pi(u),\pi(w))$. 
\end{proof}

Let us briefly comment on the estimate \eqref{eq:inv_C}. We first consider the term $\mathfrak n_i(\pi(u),\pi(v))$ where, for the sake of exposition, we fix the index $\pi(u)$ and vary the index $\pi(v)$ in a given range of values. This is equivalent to fixing the pair $(\ell, m)$ (associated to $\pi(u)$) and varying the pair $(i,j)$ (associated to $\pi(v)$) in a given bounded subset $\mathcal{B}$ of $\mathbb{N}^2$. As the quantity $\mathfrak n_i(\pi(u),\pi(v))=\vert \ell-i\vert + \vert m-j\vert -i$ represents a (modified) $\ell^1$-distance between  the points $(\ell, m)$ and $(i,j)$, it follows that varying this latter point can yield (depending on the choice of $\mathcal{B}$) a non-monotone behavior for $\mathfrak n_i(\pi(u),\pi(v))$ (see Figure~ \ref{Fig:distance}). One can conclude similarly on the non-monotone behavior of the denominator $\mathfrak n_i(\pi(u),\pi(v^*))$ in \eqref{eq:inv_C}, hence justifying the observed oscillatory behavior of the entries of $\GC$ (and hence of $G^{p}$).

\begin{figure}[t!]
\begin{center}
\includegraphics[width=0.30\textwidth]{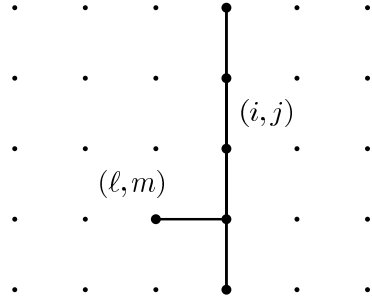}
\end{center}
\caption{An example of the non-monotone behavior of $\mathfrak n_i(\pi(u),\pi(v))$. We fix the pair $(\ell,m)$ (associated to $\pi(u)$) and vary $(i,j)$ (associated to $\pi(v)$). Moving vertically (bottom-up) the quantity $\vert \ell-i\vert + \vert m-j\vert -1$ first decreases and then increases.}\label{Fig:distance}

\end{figure}

%---------------------------------------------------------------------------------------------------
\subsubsection{Quasi-orthonormalization}
%\footnote{¤Do we need to comment on the fact that in practice to build thresholding strategies we do not resort to Proposition \ref{prop:G-decay}?}
The above documented features suggest us to invoke  procedures to wipe-out from ${G}$ a
large portion of non-zero entries, without significantly modifying the properties of the OBS basis.
A realistic approach will lead us to modify only an upper-left section of the infinite-dimensional matrix $G$ (corresponding to a certain maximal polynomial degree) and leave the rest of $G$ unchanged. We will consider and compare various strategies for compressing the chosen section of that matrix.

In all cases we will use the following notation: $G_t$ will indicate the matrix obtained from $G$ by setting to zero a certain finite set of entries, $E:=G_t-G$ will be the matrix measuring the truncation quality. We assume
${\rm diag}(G_t)= {\rm diag}(G)$, so that ${\rm diag}(E)=0$. Finally, we introduce the matrix 
\begin{equation}\label{eq:defNOBS:1}
{S}_\phi = {G}^T_t {S}_\eta {G}_t 
\end{equation}
which we interpret as the stiffness matrix associated to the modified BS basis defined in analogy to\eqref{eq:defOBS} as 
\begin{equation}\label{eq:defNOBS}
\phi_k = \sum_{m \in \mathcal{M}_t(k)}  g_{mk} \eta_m \;
\end{equation} 
where $\mathcal{M}_t(k)=\{m: E_{mk}=0\}$. This forms a new basis in $H^1_0(\Omega)$ (it is a basis since $k \in \mathcal{M}_t(k)$ and $g_{kk}\not =0$), that will be
termed a {\sl nearly-orthonormal Babu\v ska-Shen basis} (NOBS basis).

Let $D_{\phi} = {\rm diag}\,S_\phi$. We want to find a strategy to build $G_t$ such that 
the eigenvalues $\lambda$ of the problem
\begin{eqnarray}\label{eqn:maineigpb}
S_\phi x = \lambda D_{\phi} x
\end{eqnarray}
are close to one and bounded from above and away from 0 independently of the polynomial degree.
To this end the following result provides a sharp limitation on the eigenvalues in terms of the error matrix $E$. In the following, we employ the matrix norm induced by the Euclidean norm for vectors.

\begin{proposition}\label{prop:eig-bound}
Let ${L}:={G}^{-T}$ be the lower-triangular Cholesky factor of ${S}_\eta$ and $E=G_t-G$.
Assume that $\|L^TE \| < 1$. Then the eigenvalues 
$\lambda$ of \eqref{eqn:maineigpb} satisfy
$$
\frac{(1-\|L^TE\|)^2}{1+\max_i \|(L^TE)_{:,i}\|^2} \le \lambda \le (1+\|L^T E D_{\phi}^{-\frac 1 2}\|)^2.
$$
\end{proposition}

\begin{proof}
We recall that $S_{\eta}=LL^T$ and  $G_t=G+E$ from which it follows
$$S_{\phi}= (G^T+E^T) S_{\eta} (G+E) = (I+L^TE)^T (I+L^TE).$$
For $x\ne 0$, we write
$S_\phi x = \lambda  D_{\phi} x$. Multiplying by  $x^T$ we get
$x^T S_\phi x = \lambda  x^T D_{\phi} x$ with
\begin{equation}\label{aux:0}
x^T D_{\phi} x  \le (1+\max_i \|(L^T E)_{:,i}\|^2) x^Tx
\end{equation}
and, denoting by $\sigma_{\min}(A)$ the smallest singular value of $A$,
\begin{equation}\label{aux:1}
x^T S_\phi x \ge  (\sigma_{\min}( (I+L^TE)) )^2 \ge (1-\|L^T E\|)^2 .
\end{equation}
This gives the lower bound. In order to prove the upper bound we proceed as follows. For $y^T y = 1$, we have
$$
y^T D_{\phi}^{-\frac 1 2} S_\phi  D_{\phi}^{-\frac 1 2}y \le  
(\sigma_{\max}( (I+L^TE)D_{\phi}^{-\frac 1 2}) )^2  \le (1+\|L^T ED_{\phi}^{-\frac 1 2}\|)^2 .
$$

\end{proof}

\begin{corollary}\label{cor:bound-eigs}
Assume that $\|L^TE \| < 1$. Then the eigenvalues 
$\lambda$ of \eqref{eqn:maineigpb} satisfy
$$
\frac{(1-\|L^TE\|)^2}{1+\|(L^TE)\|^2} \le \lambda \le \frac{1}{(1-\|(L^TE)\|)^2}.
$$
\end{corollary}
\begin{proof}
Observing that $\|(L^T E)_{:,i}\|= \|L^T E e_i\|\leq \| L^TE \|$, being $e_i$ the $i$-th element of the canonical basis, yields the lower bound. On the other hand, taking $x=e_i$ in \eqref{aux:1}, gives
$$ \|L^T ED^{-\frac 1 2}\| \leq \frac{\|L^T E\|}{ 1- \|L^TE\|}$$
from which the upper bound easily follows.

\end{proof}

\begin{remark}
If $\sigma_{\min}(D_{\phi}^{-\frac 1 2}) \ge  \|L^TED_{\phi}^{-\frac 1 2}\|$, 
the lower bound can be sharpened as follows 
$$
\lambda \ge (\sigma_{\min}( (I+L^T E)D_{\phi}^{-\frac 1 2}))^2 \ge
(\sigma_{\min}(D_{\phi}^{-\frac 1 2}) -  \|L^TED^{-\frac 1 2}\|)^2 
$$
with $\sigma_{\min}(D_{\phi}^{-\frac 1 2}) = \min_i (D_{\phi}^{-\frac 1 2})_{ii}$.
\end{remark}

\begin{remark}
{\rm
It is easy to prove that $L^TE$ is a banded matrix whose bandwidth  
depends on the bandwidth of $L^T$ and on the bandwidth $\ell$ of $G_t$.
In particular, for $i\ne j$ we have
$$
0 = (L^T G)_{ij} = (L^TG_t + L^T E)_{ij} = 
(L^TG_t)_{ij} + (L^T E)_{ij},
$$
or, equivalently,  $|(L^T E)_{ij}|=|(L^TG_t)_{ij}|$.
Hence, if  $i,j$ are such that  $(L^TG_t)_{ij}$ is nonzero away from the nonzero band of $L^TG_t$, 
then the corresponding element of 
$L^TE$ is zero as well. Moreover, one can prove that for $j-\ell < i$  we have 
$(L^TE)_{ij}=0$. 
}
\end{remark}

%-----------------------------------------------------------------------------------------------
\subsection{Compressing the sections $G^{p}$}
%-----------------------------------------------------------------------------------------------

Hereafter, we indicate how to efficiently build compressed versions $G^{p}_t$ of the sections $G^{p}$ of the matrix $G$ for increasing values of the maximal polynomial degree $p$. 
One of the crucial quantities in the subsequent discussion will be the {\em compression ratio}
$$
r=\frac{{\sf nnz}(G_t^{p})}{{\sf nnz}(G^{p})} \;,
$$
where ${\sf nnz}(A)$ denotes the number of non-zero elements of the matrix $A$.

Proposition \ref{prop:eig-bound} suggests to build $G^{p}_t$ in such a way that  
\begin{equation}\label{eq:LTE-tol}
\| L^TE \| \leq tol_G 
\end{equation} 
is fulfilled, once a tolerance $tol_G <1$ has been fixed. This 
rigorously guarantees the achievement of our target, namely that all eigenvalues of (\ref{eqn:maineigpb}) are bounded with their reciprocals independently of
the polynomial degree. 
Note that both $L^T$ and $E$ are infinite dimensional matrices; however, the elements of $E$ are certainly zero out of a finite-dimensional section, since we
modify $G$ only within a section $G^{p}$. As a consequence, the quantity $\|L^TE\|$ is indeed computable, by considering the corresponding finite dimensional section of $L^T$.

In order to fulfill (\ref{eq:LTE-tol}), the decay estimate (\ref{eq:benzi-tuma}) suggests to proceed ``diagonal-wise", namely to build $G^{p}_t$ from $G^{p}$ by initially retaining its main diagonal and subsequently adding the $\ell$-th diagonal, for $\ell=1,2,\ldots$ until condition  \eqref{eq:LTE-tol} is satisfied. 
Figure \ref{fig:S1-0.5}, obtained with the choice $tol_G=0.5$, illustrates the typical output of this strategy: the number of activated diagonals (left) and the compression ratio (right) are reported as functions of the polynomial degree $p$.  A close inspection reveals that both quantities stabilize around constant values. This implies that the number of nonzero entries of $G_t^{p}$ needed to ensure \eqref{eq:LTE-tol} by this strategy grows significantly with $p$; note in particular the large
value of $r$ (only slightly less that $50\%$), a clear indication of the low efficiency of the procedure. 

\begin{figure}[t!]
\begin{center}
\includegraphics[width=.50\textwidth]{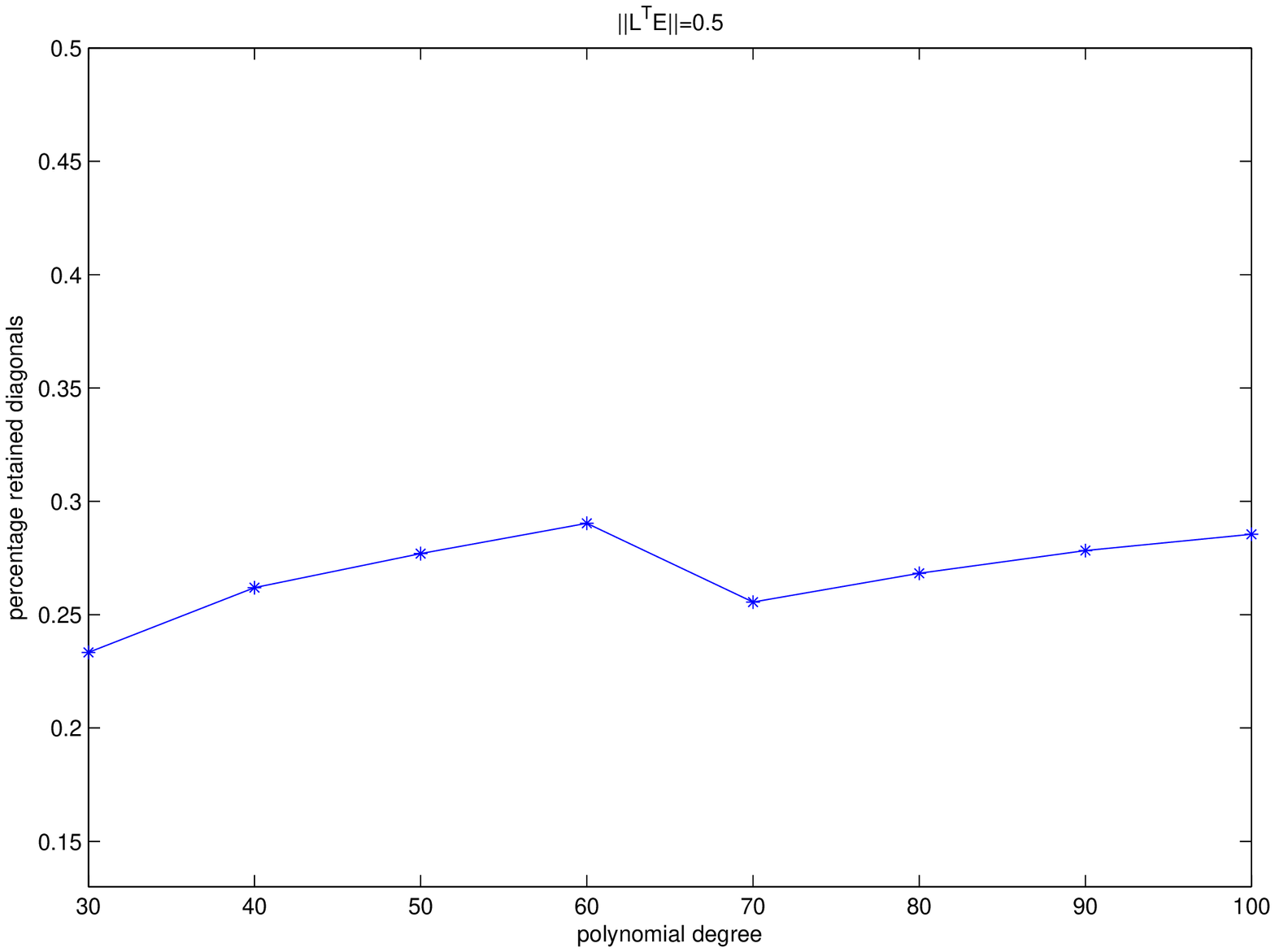}
\includegraphics[width=.49\textwidth]{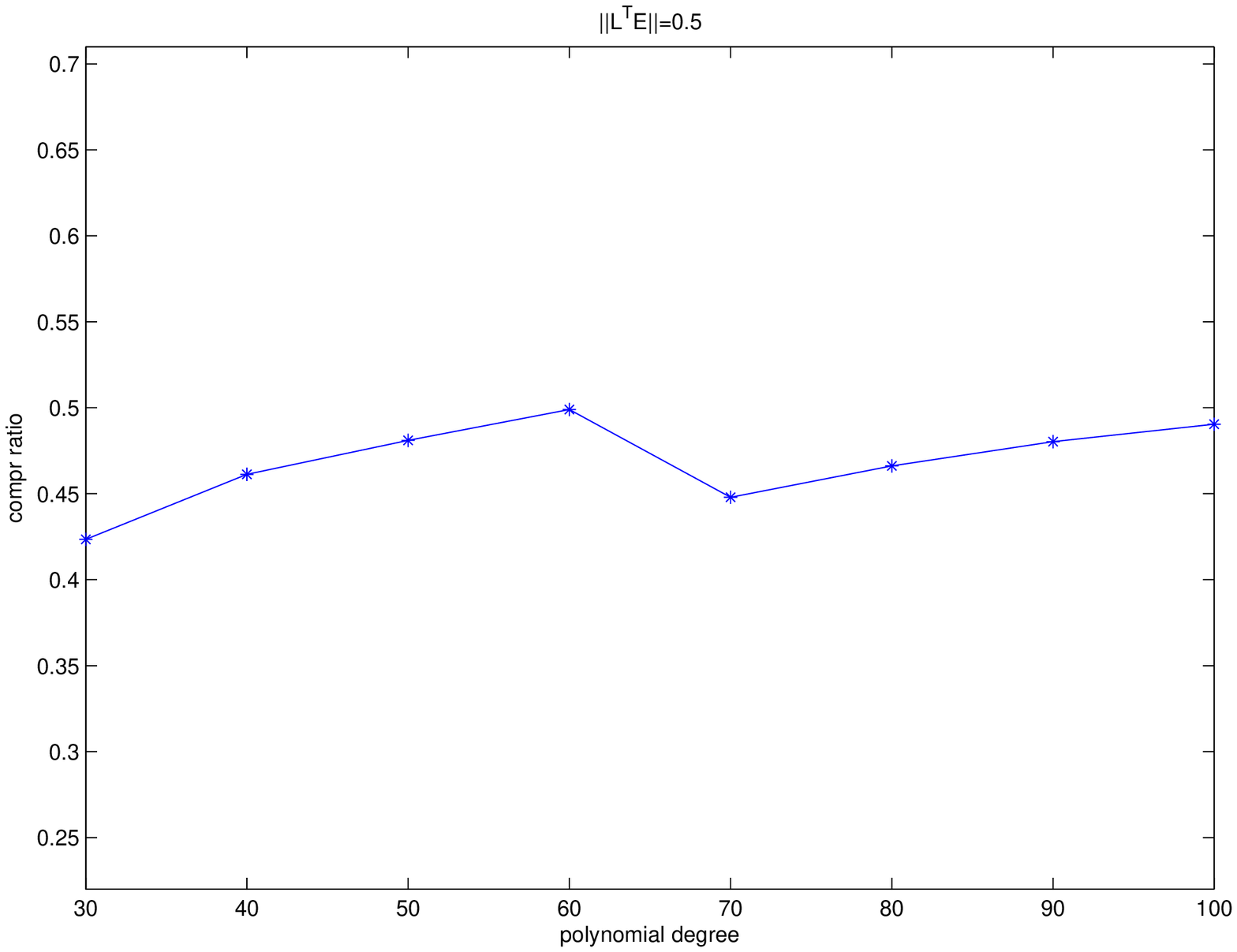}
\end{center}
\caption{Building $G^{p}_t$ by adding subsequent diagonals of $G^{p}$:
percentage of retained diagonals of $G^{p}$ (left) and compression ratio $r$ of  $G^p_t$ (right), versus the polynomial degree}
\label{fig:S1-0.5}
\end{figure}

\smallskip
However, Proposition \ref{prop:G-decay} comes in our help, as it indicates that a more sophisticated compression strategy should be applied, than simply neglecting the farthest diagonals from the main diagonal. Thus, we are led to compressing $G^p$ via a thresholding procedure; precisely, recalling formula \eqref{eq:defOBS},
the section $G_t^{p}$  is obtained by neglecting those entries of $G^{p}$ for which
\begin{equation}\label{eq:treshOBS.1}
\frac{|g_{km}|}{g_{kk}} < t\;,
\end{equation}
where $t \in (0,1)$ is the thresholding parameter. The value of $t$ is implicitly defined by the condition that (\ref{eq:LTE-tol}) be satisfied with $\| L^TE \|$ as close 
as possible to $tol_G$. A simple bisection procedure allows one to identify such a nearly-optimal value of $t$. 

Figures \ref{fig:S2-bisection-A} and \ref{fig:S2-bisection-B}, again obtained with $tol_G=0.5$, illustrate typical outputs of this thresholding strategy.
In particular, Fig.  \ref{fig:S2-bisection-A} (left) shows that the thresholding parameter $t$, identified by the bisection procedure, decays as the polynomial degree increases. This behavior seems unavoidable in order to reach our final target; indeed, numerical experiments (not reported here) clearly indicate that if we fix the
thresholding parameter and vary $p$, not only the quantity $\| L^TE \|$ eventually becomes larger than any  fixed $tol_G<1$, but the smallest eigenvalue of the resulting stiffness matrix $S_\phi$ decays to $0$.
The observed behavior of $t$ implies that entries of $G$ that were set to 0 in $G^p_t$ for a lower value of $p$, may subsequently be included in  $G^p_t$ for higher
values of $p$. This phenomenon is well-documented in Fig.  \ref{fig:S2-bisection-A} (right); we fixed one of the ``slow decaying'' columns of $G$, precisely
column $98$ already considered in Figure \ref{fig:slow-fast} (left), and we counted the number of nonzero elements in that column of $G^p_t$: the growth with
$p$ is apparent. This indicates that there are no (upper left) sectors of $G^{p}_t$ that remain unchanged while furtherly increasing $p$; the construction of $G^{p}_t$ is ``global" and may involve all relevant columns. Note, however, that for the same range of $p$ as in Fig.  \ref{fig:S2-bisection-A} (right), we observed that 
the number of nonzero elements of the ``fast decaying'' column $105$ of $G^p_t$ remains fixed to 1. 

 Although we cannot expect the number of non-zero entries in $G^p_t$ to be proportional to the dimension of the matrix, Fig.  \ref{fig:S2-bisection-B} (left) shows
 that the compression ratio is decaying, but at a very slow rate with $p$ and, more importantly, it is more than one order of magnitude smaller than the compression rate guaranteed 
 by the ``diagonal-wise" strategy (compare with Fig. \ref{fig:S1-0.5} (right)). One example of compressed matrix $G^p_t$ produced in this manner (for $p=100$) is
 shown in Fig.  \ref{fig:S2-bisection-B} (right). In addition, the minimal and maximal eigenvalues of the resulting stiffness matrices $S_\phi=S^p_\phi$ exhibit a very 
 moderate deviation from the optimal value 1; this is documented in Fig. \ref{fig:extreme-eigenvalues}, which provides a quantitative insight of the upper and
 lower bounds guaranteed by Corollary \ref{cor:bound-eigs}.
 
 In conclusion, the thresholding strategy here discussed appears to guarantee the achievement of our target with a good efficiency for all values of the polynomial degree $p$ relevant in practical implementations. 

\begin{figure}[t!]
\begin{center}
\includegraphics[width=.49\textwidth]{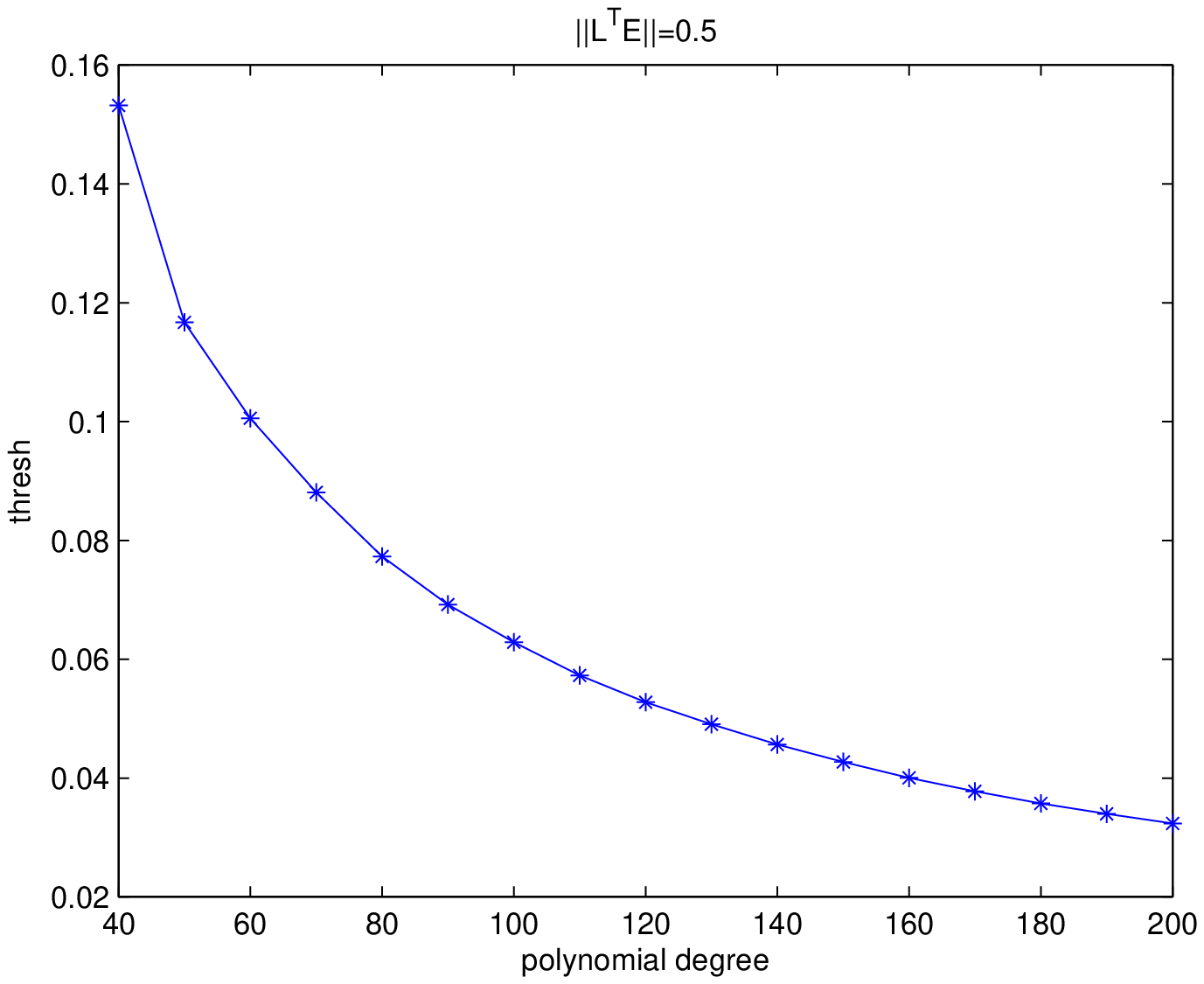}
\includegraphics[width=.50\textwidth]{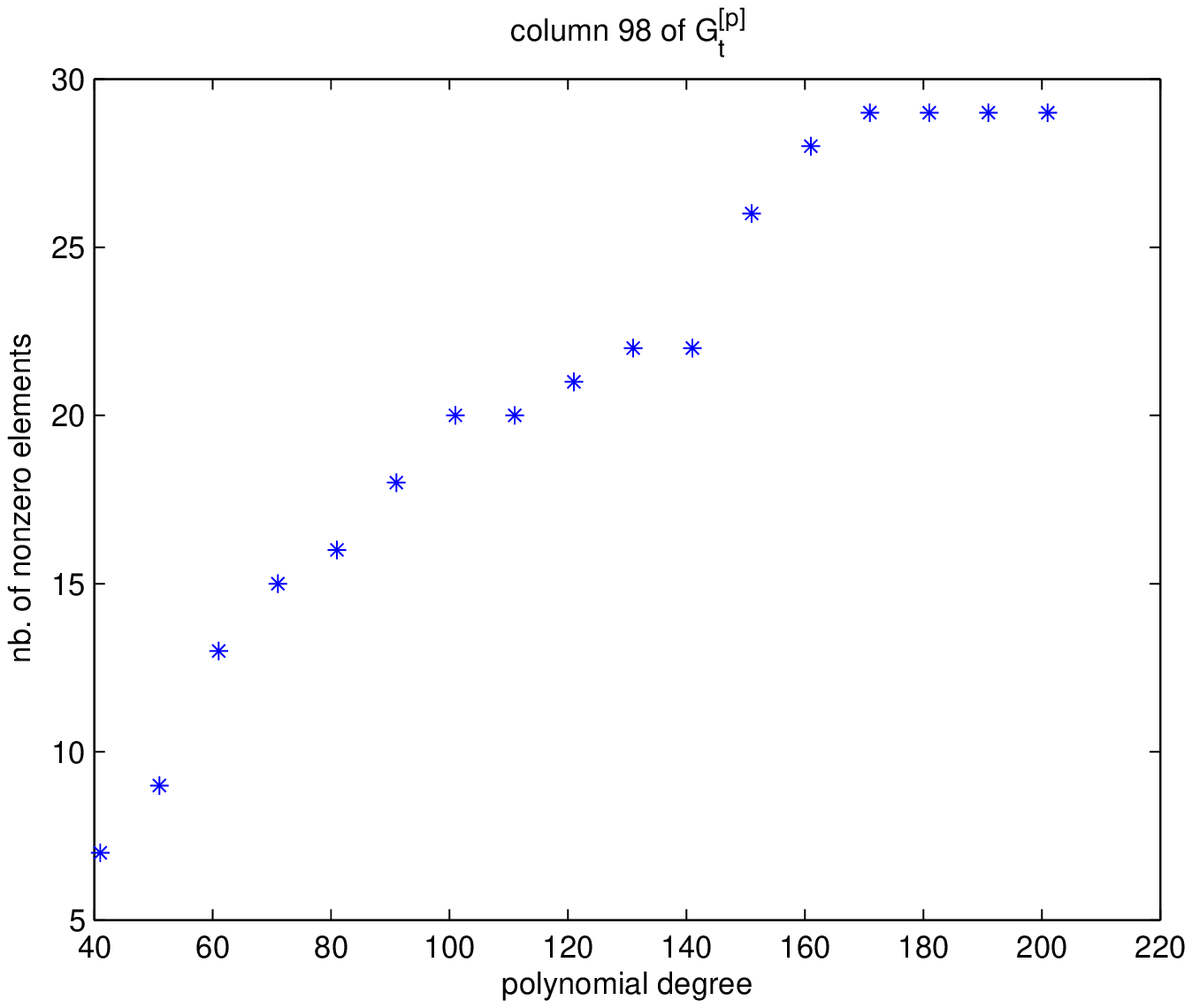} 
\end{center}
\caption{Building $G^{p}_t$ by a thresholding procedure:
thresholding parameter $t$ (left) and number of nonzero elements of  a ``slow'' decaying column of $G_t$ (right), versus the polynomial degree}
\label{fig:S2-bisection-A}
\end{figure}

\begin{figure}[t!]
\begin{center}
\includegraphics[width=.50\textwidth]{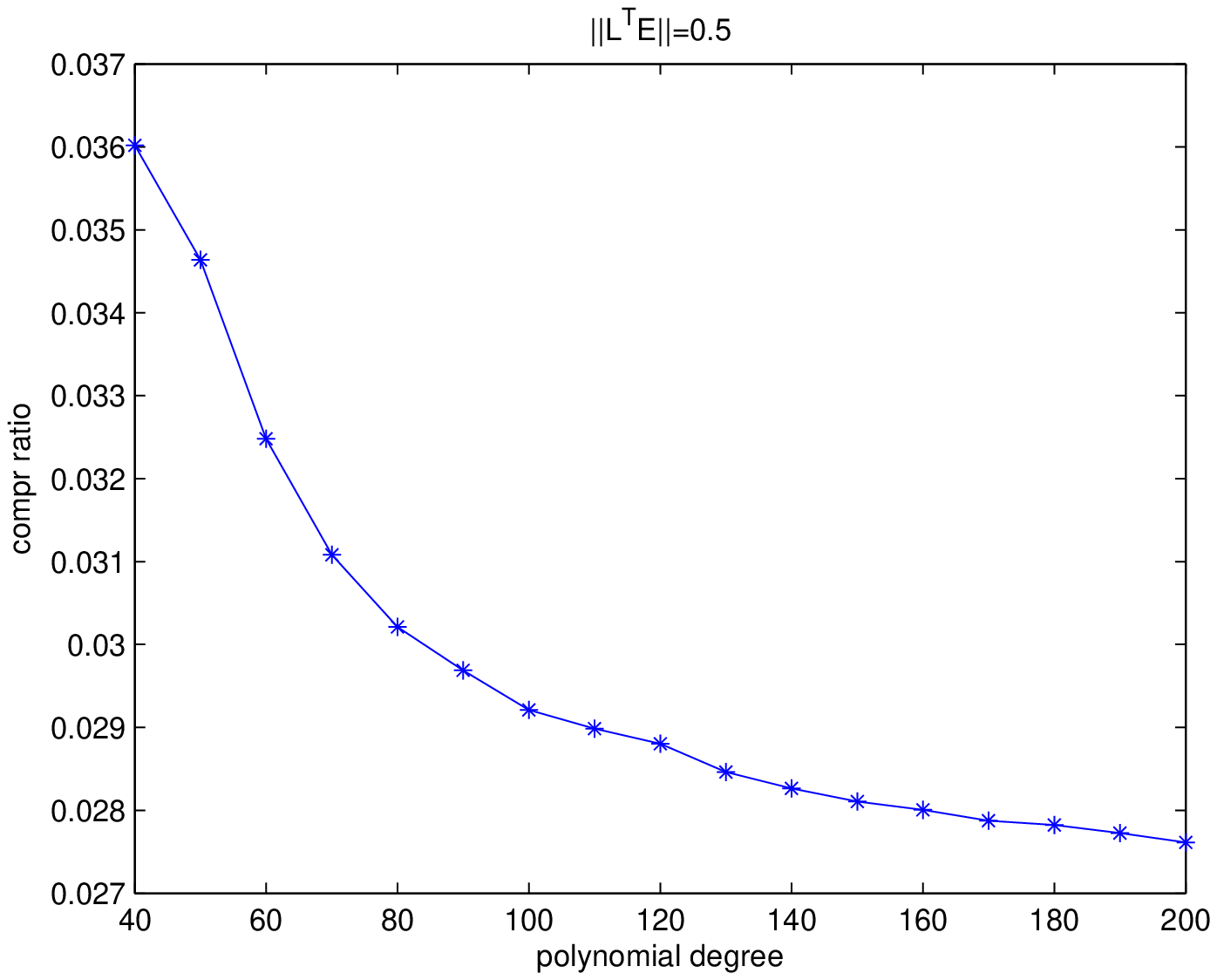}
\includegraphics[width=.34\textwidth]{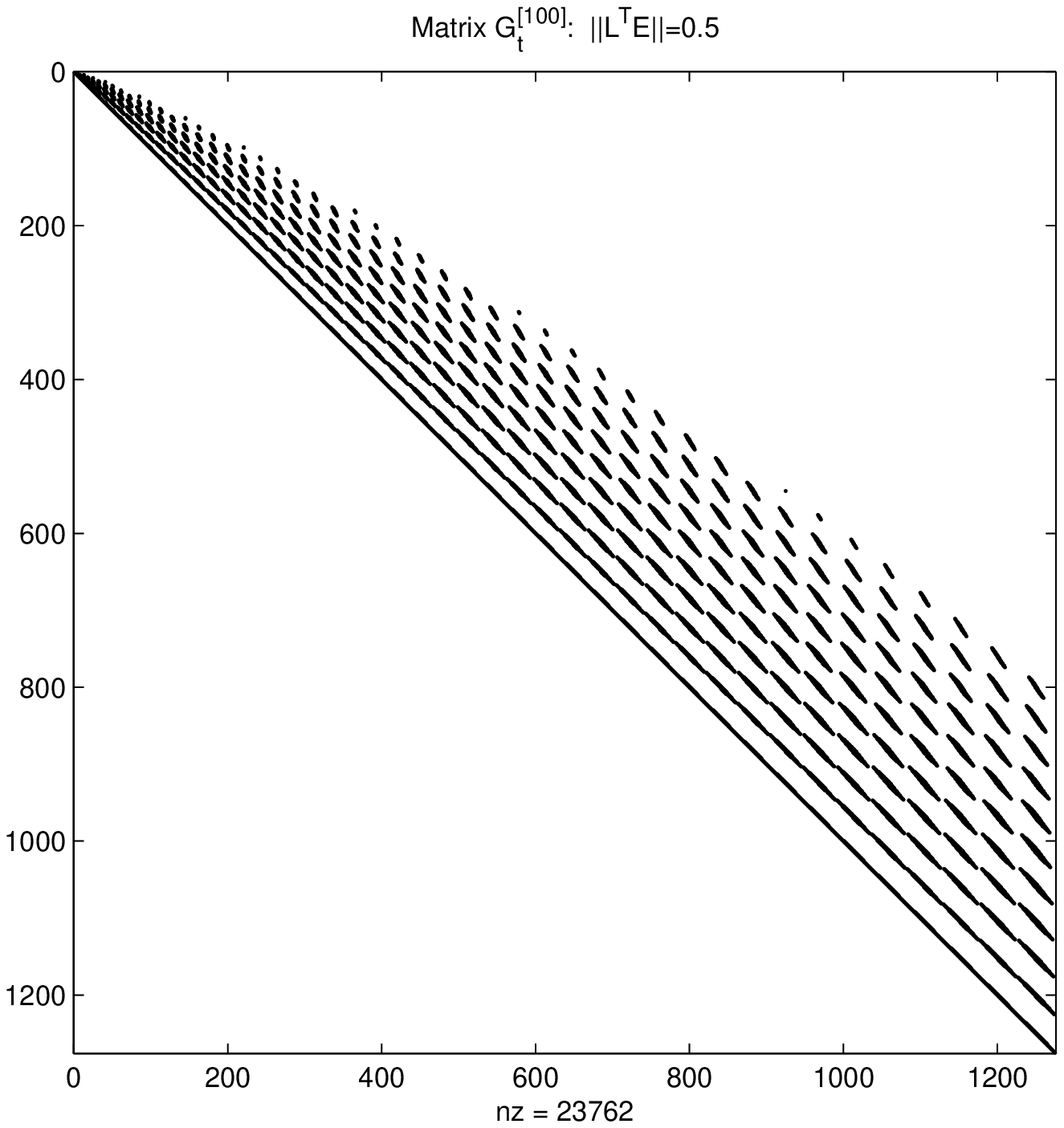}
\end{center}
\caption{Building $G^{p}_t$ by a thresholding procedure:
compression ratio $r$ of  $G^p_t$  versus the polynomial degree (left), sparsity pattern of the matrix $G^{100}_t$ (right)}
\label{fig:S2-bisection-B}
\end{figure}

\begin{figure}[t!]
\begin{center}
\includegraphics[width=.45\textwidth]{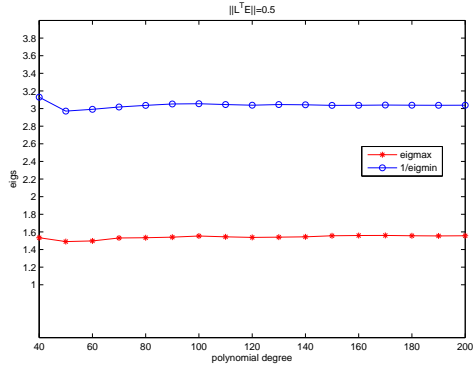}
\end{center}
\caption{Extreme eigenvalues of $S_{\phi}$  versus the maximum polynomial degree.}
\label{fig:extreme-eigenvalues}
\end{figure}

%-----------------------------------------------------------------------------------------------
\subsection{Norm representation}

From now on, we assume that we work in $H^1_0(\Omega)$ with a Riesz basis $\phi=\{\phi_k \, : \, k \in {\cal K}\}$ given by (\ref{eq:defNOBS}), where the matrix $G_t$ is built according to the thresholding strategy presented above, with a fixed value of tolerance $tol_G<1$ and a fixed polynomial degree $p_{\text{max}}$; 
precisely, the upper-left section $G^{p_{\text{max}}}$ of $G$ is
replaced by  $G^{p_{\text{max}}}_t$ while the rest of $G$ is unchanged.
Obviously, dealing with such a basis is computationally efficient only if the adaptive algorithm will reach the prescribed accuracy by activating only basis functions having polynomial degree $p \leq p_{\text{max}}$. We will implicitly make this assumption in the sequel.

Thus, if $v \in H^1_0(\Omega)$ admits the expansion 
$v = \sum_{k \in {\cal K}} \hat{v}_k \phi_k$ and 
$\hat{v}$ denotes the vector collecting its coefficients $\hat{v}_k$, we have
\begin{equation}\label{eq:propNOBS.3}
\Vert v \Vert_{H^1_0({\Omega})}^2  = \hat{v}^T {S}_\phi \, \hat{v} \simeq \hat{v}^T {D}_\phi \, \hat{v}
= \sum_{k \in {\cal K}} |\hat{v}_k|^2 d_k=:\Vert v \Vert_\phi^2 \;,
\end{equation}
where ${D}_\phi :=  \mathrm{diag} \, S_\phi$ is the diagonal matrix with diagonal elements  
$d_k=(S_\phi)_{k,k}$. 
Correspondingly, any element $f \in H^{-1}(\Omega)$ can be expanded along the 
{\sl dual nearly-orthonormal Babu\v ska-Shen basis} $\phi^*=\{\phi_k^*\}$ as 
$f = \sum_{k \in {\cal K}} \hat{f}_k \phi_k^*$, with $\hat{f}_k = \langle f,\phi_k \rangle$,
yielding the dual norm representation
\begin{equation}\label{eq:propNOBS.4}
\Vert f \Vert_{H^{-1}(\Omega)}^2  \ \simeq \  \sum_{k \in {\cal K}} |\hat{f}_k|^2 d_k^{-1}=:\Vert v \Vert_{\phi*}^2\;.
\end{equation}
Note that each coefficient $\hat{f}_k$ can be efficiently computed from the values of $f$ on the elements of the
standard BS basis, via (\ref{eq:defNOBS}):
\begin{equation}\label{eq:propNOBS.5}
\hat{f}_k = \sum_{m \in \mathcal{M}_t(k)} g_{km} \langle f,\eta_m \rangle\;.
\end{equation}
\medskip
\noindent {\bf Notation.} 
For future references, we introduce the constants of the norm equivalence in (\ref{eq:propNOBS.3}), i.e., we assume that the constants
$\beta_*$ and $\beta^*$ are such that
\begin{equation}\label{eq:propNOBS.7}
\beta_* \Vert v \Vert_{H^1_0(\Omega)} \leq \Vert v \Vert_\phi \leq \beta^* \Vert v \Vert_{H^1_0(\Omega)}
\qquad \forall v \in H^1_0(\Omega) \;,
\end{equation}
which implies
\begin{equation}\label{eq:propNOBS.8}
\frac1{\beta^*} \Vert f \Vert_{H^{-1}(\Omega)} \leq \Vert f \Vert_{\phi^*} \leq 
\frac1{\beta_*} \Vert f \Vert_{H^{-1}(\Omega)}
\qquad \forall f \in H^{-1}(\Omega) \;.
\end{equation}

Moreover, the $\phi$-norm of any $v\in H^1_0(\Omega)$ is equivalent to the $\ell^2$-norm of the vector $\hat{v}$ of its coefficients, i.e., there exist two constants $0<d_*<d^*$ such that 
\begin{equation}\label{eq:equiv3}
d_*\| \hat{v}\|^2_{\ell^2} \leq \| {v}\|^2_{\phi} \leq d^* \| \hat{v}\|^2_{\ell^2} \qquad \forall v \in H^1_0(\Omega) \ .
\end{equation}
Indeed, employing \eqref{aux:0} and \eqref{aux:1}  (under our assumption $\|L^TE\| < 1$) it is immediate to prove that it holds 
$$d_*:=(1-\|L^TE\|)^2\leq d_k\leq 1+ \|L^TE\|^2=:d^*\qquad  \forall k\in \mathcal{K}. $$

%----------------------------------------------------------------------------------------------
\section{The differential problem and its algebraic representation}\label{sec:gen}
%------------------------------------------------------------------------------------------------------

We now consider the elliptic problem
\begin{equation}\label{eq:four03}
\begin{cases} 
{\mathcal L}u=-\nabla \cdot (\nu \nabla u)+ \sigma u = f & \text{in } \Omega \;, \\
u=0 &  \text{on } \partial\Omega \;, 
\end{cases} 
\end{equation}
where $\nu$ and $\sigma$ are sufficiently smooth real coefficients satisfying 
$0 < \nu_* \leq \nu(x) \leq \nu^* < \infty$ and $0 \leq \sigma(x) \leq \sigma^* < \infty$
in $\Omega$; let us set $\alpha_* = \nu_*$ and $\alpha^* = \max(\nu^*, \sigma^*)$.
Assuming $f \in H^{-1}(\Omega)$, we formulate this problem variationally as
\begin{equation}\label{eq:four.1}
u \in V \ \ : \quad a(u,v)= \langle f,v \rangle \qquad \forall v \in  V \;,
\end{equation}
where $a(u,v)=\int_\Omega \nu \nabla u \cdot \nabla v + \int_\Omega \sigma u v$. We denote by 
$\tvert v \tvert = \sqrt{a(v,v)}$
the energy norm of any $v \in V$, which satisfies 
\begin{equation}\label{eq:four.1bis}
\sqrt{\alpha_*}  \Vert v \Vert_{H^1_0(\Omega)}  \leq \tvert v \tvert \leq 
\sqrt{\alpha^*}  \Vert v \Vert_{H^1_0(\Omega)}  \;.
\end{equation}

%%%%%%%%%%%%%%%%%%%%%%%%%%%%%%%%%%%%%%%%%%%%%%%%%
%\subsection{Algebraic representation and properties of the stiffness matrix}
%%%%%%%%%%%%%%%%%%%%%%%%%%%%%%%%%%%%%%%%%%%%%%%%%

Let us identify the solution $u = \sum_k \hat{u}_k \phi_k$ of Problem (\ref{eq:four.1})
with the vector ${u}=(\hat{u}_k)_{k\in {\cal K} }$ 
of its nearly-orthonormal Babu\v ska-Shen (NOBS) coefficients. Similarly, let us identify
the right-hand side $f$ with the vector ${f}=(\hat{f}_\ell)_{\ell \in {\cal K}}$ of its dual NOBS coefficients.
Finally, let us introduce the semi-infinite, symmetric and positive-definite stiffness matrix 
\begin{equation}\label{eq:four100}
{A_\phi}=(a^\phi_{\ell k})_{\ell,k \in {\cal K}} \qquad \text{with} \qquad 
a^\phi_{\ell k}= a(\phi_k,\phi_\ell)\;.
\end{equation}
Then, Problem (\ref{eq:four.1}) can be equivalently written as
\begin{equation}\label{eq:four110}
{A_\phi} {u} = {f} \;,
\end{equation}
where, thanks to the previous assumptions and the norm equivalences \eqref{eq:propNOBS.7}-\eqref{eq:equiv3}, ${A}_\phi$ defines a bounded invertible operator 
in  {$\ell^2({\cal{K}})$}. 

The rest of this section will be devoted to prove that if the operator coefficients  $\nu$ and $\sigma$ are real analytic  in a neighborhood of $\bar \Omega= [0,1]^2\subset\mathbb{C}\times \mathbb{C}$ (which implies that the rate of decay of their Legendre coefficients is exponential), then the stiffness matrix $A_\phi$ belongs to a certain exponential class, i.e., its entries exponentially decay away from the diagonal. This property will be crucial in studying the optimality properties of the subsequent adaptive algorithm. After introducing the classes of  exponentially decaying matrices and some useful properties related to them, we will obtain the claimed result by using the relation
\begin{equation}\label{relation-matrix}
A_\phi=G_t^T A_\eta G_t
\end{equation}
together with suitable exponential decay properties of the factors $A_\eta$ and $G_t$.   

\begin{definition}[{Regularity classes for $A$}]\label{def:class.matrix}
A matrix ${A}$ is said to belong to the {exponential} class 
${\mathcal D}_e(\gamma)$ if there exists a constant $c_\gamma>0$ 
such that its elements satisfy
\begin{equation}\label{eq:four170}
| a_{mn} | \leq  c_\gamma e^{-\gamma\| m - n \|_{\ell^{1}}}\;  \qquad m, n  \in{\cal K} \;.
\end{equation}
\end{definition}
\begin{property}[{Inverse of $A$}] \label{prop:inverse.matrix-estimate}
If ${A} \in {\mathcal D}_e(\gamma)$ and it is invertible then 
${ A}^{-1}\in{\mathcal D}_e(\bar{\gamma})$, for some
$\bar{\gamma} \in (0,\gamma]$ 
\end{property}
\begin{proof}
See \cite[Proposition 2]{Jaffard:1990}.
\end{proof}

\noindent For any integer $J \geq 0$, let ${A}_J$ denote {the following
symmetric truncation of the matrix ${A}$}
\begin{equation}\label{eq:trunc-matr}
({A}_J)_{\ell k}=
\begin{cases}
a_{\ell k} & \text{if } \|\ell-k\|_{\ell^{1}} \leq J \;, \\
0 & \text{elsewhere.}
\end{cases}
\end{equation}
Then, we have the following well-known results (see, e.g., \cite[Property 2.4]{CNV:mathcomp}).
\begin{property}[Truncation]\label{prop:matrix-estimate}
If ${A} \in {\mathcal D}_e(\gamma)$ then there exists a constant 
$C_{{A}} $ such that 
\[
\Vert {A}-{{A}}_J \Vert \leq
\psi_{{A}}(J,\gamma):=C_{{A}} {\rm e}^{-\gamma J}
\]
{for all $J\ge0$.}
Consequently, under the assumptions of Property \ref{prop:inverse.matrix-estimate},
one has
\begin{equation}\label{eq:trunc-invmatr-err}
\Vert {A}^{-1}-({A}^{-1})_J \Vert \leq 
\psi_{{A}^{-1}} (J,\bar{\gamma})
\end{equation}
where we {let $\bar\gamma$ be defined in Property \ref{prop:inverse.matrix-estimate}.}
\end{property}

We now state the basic assumption on the coefficients of the operator $\mathcal{L}$.
\begin{assumption}\label{ass:coeff}
Let $\nu(x)=\sum_{k\in{\cal K}} \nu_k L_k(x)$ and $\sigma(x)=\sum_{k\in{\cal K}} \sigma_k L_k(x)$, resp.,  be the multidimensional Legendre expansions of the operator coefficients $\nu$ and $\sigma$, resp. (with $L_k(x):=L_{k_1}(x_1)L_{k_2}(x_2)$). There exist  $\gamma>0$ and a positive constant $C_\gamma$ only depending on $\gamma$ such that 
$$
\vert \nu_k \vert , \  \vert \sigma_k \vert  \leq C_\gamma e^{-\gamma \|k\|_{\ell^{1}}}\qquad \forall k\in{\cal K}.$$
\end{assumption}

\begin{lemma}[Exponential decay of $A_\eta$]\label{lm:L}
Let $\{\eta_k\}_{k\in {\cal K}}$  be the tensorized BS basis functions and ${A_\eta}=(a^\eta_{\ell k})_{\ell,k \in {\cal K}}$ with $a^\eta_{\ell k}= a(\eta_k,\eta_\ell)$ be the stiffness matrix associated to the operator ${\mathcal L}$. Under Assumption \ref{ass:coeff}, it holds
\begin{equation}
 \vert a^\eta_{mn} \vert \leq C e^{-\gamma \|n-m\|_{\ell^{1}}} \qquad \forall n,m\in{\cal K}\ ,
\end{equation}
where $C$ is a constant only depending on $\eta$.

\end{lemma}
\begin{proof}
Due to the highly technical nature of the proof, we postpone it to the Appendix.
\end{proof}

The proof of the exponential decay of $G_t$ relies on the following intermediate result.

\begin{lemma}[Exponential decay of $S^{-1}_\eta$]\label{lm:invS}
Let ${S}_\eta$ be the stiffness matrix of the
tensorized Babu\v ska-Shen basis with respect to the $H^1_0(\Omega)$-inner product. Then there exists $\hat \gamma>0$ such that $S_\eta^{-1}\in {\mathcal D}_e(\hat \gamma)$.
\end{lemma}
\begin{proof}
We preliminary observe that, thanks to \eqref{eq:orthogonal}$, S_\eta$ is a banded matrix with half bandwidth equal to $2$, once we endow the index set ${\cal K}$ with the $\ell^{1}$-metric, i.e., $(S_\eta)_{mk}=0$ if 
$\|k-m\|_{\ell^{1}}>2$. To conclude it is sufficient to apply Property \ref{prop:inverse.matrix-estimate} to $S_\eta$ (which is banded, thus trivially with exponential decay). 
\end{proof}

\begin{lemma}[Exponential decay of $G$ and $G_t$]\label{lm:G}
Let $G$ be the semi-infinite matrix satisfying equation \eqref{eq:propOBS.1}. Then there exists $\tilde \gamma>0$ such that  
$G\in {\mathcal D}_e(\gamma)$. Moreover, let $G_t$ denote the matrix obtained from $G$ by setting to zero a certain finiite (or infinite) set of entries. Then $G_t$ belongs to the same exponential class of $G$, i.e. $G_t \in {\mathcal D}_e(\tilde \gamma)$.
\end{lemma}
\begin{proof}
We proceed extending the idea of \cite[Theorem 4.1] {Benzi-Tuma:2000} to the present infinite-dimensional setting (see also \cite{Hall-Jin:2010,Strohmer-et-al:2013} for similar results). Recall that $S_\eta$ has been normalized so that $(S_\eta)_{i,i,}=1$ for all $i\in{\cal K}$. Here $S_\eta=L L^T$ is the Cholesky factorization of $S_\eta$ (with $L$ lower triangular infinite dimensional matrix). Using the fact that $S_\eta$ is a banded matrix (with unitary band) once the index set ${\cal K}$ is endowed with the $\ell^{1}$ metric and noticing that it holds $G=L^{-T}$ we get \begin{equation}
G_{i j}=\sum_{k=j}^{j+1} (S^{-1}_{\eta})_{i k} \ell_{k j} \qquad i,j\in {\cal K}. 
\end{equation}
Employing the exponential decay of $S_{\eta}^{-1}$ (see Lemma \ref{lm:invS}) and recalling that the normalization of $S_\eta$ implies $\vert \ell_{i j}\vert \leq 1$ for all $i,j\in {\cal K}$, we obtain  
\begin{eqnarray}
|G_{i j}| &\le& \sum_{k=j}^{j+1} \vert (S^{-1}_{\eta})_{i k}\vert \vert  \ell_{k j} \vert \le C \sum_{k=j}^{j+1} e^{-\hat \gamma \| i-k\|_{\ell^{1}}}\nonumber\\
&\leq& C \sum_{q\geq \| i-j\|_{\ell^{1}}} \text{card}(K(q)) e^{-\hat \gamma q} 
\end{eqnarray}
where $K(q):=\{ k \in [j,j+1]:\quad \|i-k\|_{\ell^{1}}=q\}\subset {\cal K}$. Note that as the index subset $K(q)$ inherits the ordering of ${\cal K}$ (see Fig.~\ref{Fig:ordinamenti-A}), the notion of interval $[j,j+1]$ entering into the definition of $K(q)$ has to be intended consistently with this ordering. It is immediate to verify that it holds $ \text{card}(K(q)) \leq \tilde{C} q$ for a positive constant $\tilde{C}$ independent of $q$. Then it follows, for some $\tilde \gamma < \hat \gamma$. 
\begin{equation}
|G_{i j}| \le  C^\prime \sum_{q\geq \| i-j\|_{\ell^{1}}} e^{-\hat \gamma q}\leq C^{\prime\prime} e^{-\tilde \gamma \| i-j\|_{\ell^{1}}} \qquad i,j\in {\cal K}.
\end{equation}
The second part of the theorem is immediate.
\end{proof}

\noindent Now we are ready to establish the main result of this section.

\begin{proposition}\label{prop:decay}
Under Assumption \ref{ass:coeff} there exists $\gamma_{\cal L} \in (0,\gamma]$ such that $A_\phi\in {\mathcal D}_e(\gamma_{\cal L})$.
\end{proposition}

\begin{proof}
Recall the expression of $A_\phi$ given in \eqref{relation-matrix}. 
Employing \cite[Proposition 1]{Jaffard:1990}  together with Lemmas \ref{lm:G} and \ref{lm:L} we deduce that the product $A_\phi=G_t^T A_\eta G_t$ of exponentially decaying semi-infinite matrices is exponentially decaying with a lower decay rate.
\end{proof}

%\begin{itemize}
%\item Enunciare Jaffard's result (Proposition $2$ del paper {\em Jaffard, Proprietes des matrices Ç bien localisees È pres de leur diagonale et quelques applications}).
%\item verifica (via Jaffard's result) del decadimento esponenziale dell'inversa $S^{-1}_\eta$ della matrice semi-infinita di rigidezza associata al laplaciano con BS (che ha banda unitaria se si pensano gli indici nello spazio metrico $\mathbb{N}^2$, usando la distanza $\ell^{1}$ tra gli indici).
%\item verifica del decadimento esponenziale della matrice semi-infinita $G$ (associata) estendendo la proof di Benzi-Tuma al caso infinito dimensionale.
%\item[($3$bis)] Osservare che lo stesso risultato vale banalmente per la matrice troncata $G_t$
%\item verifica decadimento esponenziale della matrice di stiffness $A_\eta$ (basis BS) con coefficienti dell'operatore variabili,  i cui coefficienti di Legendre decadono esponenzialmente. 
%\item Utilizzando il fatto che il prodotto di matrici semi-infinite con decadimento esponenziale ha decadimento esponenziale (trovare citazione opportuna) si ottiene che ${G}^T {A}_\eta {G}$ ha decadimento esponenziale.
%\item[($5$bis)]   Osservare che anche ${G}_t^T {A}_\eta {G}_t$ ha decadimento esponenziale.
%\end{itemize}

%%%%%%%%%%%%%%%%%%%%%%%%%%%%%%%%%%%%%%%%%%%%%%
\section{An adaptive algorithm with contraction properties}\label{sec:adapt-alg}
%%%%%%%%%%%%%%%%%%%%%%%%%%%%%%%%%%%%%%%%%%%%%%%
Given any finite index set $\Lambda \subset {\cal K}$, we define the subspace $V_{\Lambda} = {\rm span}\,\{\phi_k\, | \, k \in \Lambda \}$ of $V=H^1_0(\Omega)$;  we set $|\Lambda|= \rm{card}\, \Lambda$, so that $\rm{dim}\, V_{\Lambda}=|\Lambda|$. We set $\text{supp}\,  v := \Lambda$ if $v=\sum_{k\in\Lambda} \hat v_k \phi_k$
with all $ \hat v_k\not = 0$. If $v\in V$ admits the
expansion $v = \sum_{k \in {\cal K}} \hat{v}_k \phi_k $, then we define its 
projection $P_\Lambda v$ upon $V_\Lambda$ by setting
$P_\Lambda v = \sum_{k \in \Lambda} \hat{v}_k \phi_k$.
Similarly, we define the subspace $V_{\Lambda}^* = {\rm span}\,\{\phi^*_k\, | \, k \in \Lambda \}$ of $V'=H^{-1}(\Omega)$; if $v$ admits an expansion $f = \sum_{k \in {\cal K}} \hat{f}_k \phi^*_k $, then we define its  projection $P^*_\Lambda f$ upon $V^*_\Lambda$ by setting
$P^*_\Lambda f = \sum_{k \in \Lambda} \hat{f}_k \phi^*_k $.

\smallskip
Given any finite $\Lambda \subset {\cal K}$, the Galerkin approximation of \eqref{eq:four.1} is defined as
\begin{equation}\label{eq:four.2}
u_\Lambda \in V_\Lambda \ \ : \quad a(u_\Lambda,v_\Lambda)= 
\langle f,v_\Lambda \rangle \qquad \forall v_\Lambda \in V_\Lambda \;.
\end{equation}

For any $w \in V_\Lambda$, we define the residual $r(w) \in V'$ as
$$
r(w)=f-{\mathcal L}w = \sum_{k \in {\cal K}} \hat{r}_k(w) \phi^*_k \;, \qquad \text{where} \qquad 
\hat{r}_k(w) = \langle f - {\mathcal L}w, \phi_k \rangle = \langle f,\phi_k \rangle -a(w,\phi_k) \;.
$$
Then, the previous definition of $u_\Lambda$ is equivalent to the condition $P^*_\Lambda r(u_\Lambda) = 0$, i.e., $\hat{r}_k(u_\Lambda)=0$ for every $k \in \Lambda$.
By the continuity and coercivity of the bilinear form, one has 
\begin{equation}\label{eq:four.2.1}
\frac1{\alpha^*} \Vert r(u_\Lambda) \Vert_{H^{-1}(\Omega)} \leq
\Vert u - u_\Lambda \Vert_{H^1_0(\Omega)} \leq 
\frac1{\alpha_*} \Vert r(u_\Lambda) \Vert_{H^{-1}(\Omega)} \;,
\end{equation}
which in view of \eqref{eq:propNOBS.8} can be rephrased as 
\begin{equation}\label{eq:four.2.1bis}
\frac{\beta_*}{{\alpha^*}} \Vert r(u_\Lambda) \Vert_{\phi^*} \leq
\| u - u_\Lambda \|_{H^1_0(\Omega)} \leq 
\frac {\beta^*}{{\alpha_*}} \Vert r(u_\Lambda) \Vert_{\phi^*} \;.
\end{equation}
%We start considering as an error estimator the ideal one, i.e., the norm of the residual in $V'$; thanks to (\ref{eq:propNOBS.8}),
%this norm can be replaced by the $\phi^*$-norm. We thus set, for any $v \in V$,
%\begin{equation}\label{eq:four.2.2}
%\eta^2(v)=\Vert r(v) \Vert^2_{\phi^*} = \sum_{k \in \cal K} |\hat{r}_k(v)|^2 d_k^{-1}\;,
%\end{equation}
%so that (\ref{eq:four.2.1}) can be rephrased as
%\begin{equation}\label{eq:four.2.3}
%\frac{\beta_*}{\alpha^*} \eta(u_\Lambda) \leq
%\Vert u - u_\Lambda \Vert \leq 
%\frac{\beta^*}{\alpha_*} \eta(u_\Lambda) \;.
%\end{equation}
%Given any subset $\Lambda \subseteq \cal K$, we also define the quantity
%$$
%\eta^2(v;\Lambda) = \Vert P^*_\Lambda r(v) \Vert^2_{\phi^*} 
%= \sum_{k \in \Lambda} |\hat{r}_k(v)|^2\;,
%$$
%so that $\eta(v)=\eta(v;\cal K)$.
The norm of the residual is hardly computable in 
practice, since in general the residual $r(u_\Lambda)$ contains infinitely many coefficients.
Therefore, we introduce an approximation $\wtilde{r}(u_\Lambda)$ of such 
residual with finite expansion (indexed in some finite set $\tilde\Lambda\subset \mathcal{K}$). 
More precisely,  we assume there exists a feasible algorithm which for any fixed parameter $0<\delta<1$
 and for any $v$ with a finite expansion builds an approximation 
$\wtilde{r}(v)$ of $r(v)$ such that the following crucial inequality holds :
\begin{equation}\label{eq:four.2.6}  
\Vert r(v) - \wtilde{r}(v)  \Vert_{\phi^*}  \leq \delta
\Vert \wtilde{r}(v)  \Vert_{\phi^*}  \;.
\end{equation}
This implies $(1-\delta)\Vert {\wtilde{r}}(v)  \Vert_{\phi^*}  \leq \Vert {r}(v)  \Vert_{\phi^*}  \leq 
(1+\delta)\Vert \wtilde{r}(v)  \Vert_{\phi^*} $ so that, by \eqref{eq:four.2.1bis}, we obtain 
\begin{equation}\label{eq:four.2.3bis}
(1-\delta)\frac{\beta_*}{{\alpha^*}} \Vert \wtilde{r}(u_\Lambda) \Vert_{\phi^*} \leq
\| u - u_\Lambda \|_{H^1_0(\Omega)} \leq 
(1+\delta)\frac {\beta^*}{{\alpha_*}} \Vert \wtilde{r}(u_\Lambda) \Vert_{\phi^*} \;.
\end{equation} 
Therefore, we are led to use as a posteriori error estimator the quantity 
\begin{equation}\label{eq:four.2.7}
{\text{Est}}(u_\Lambda)=\Vert \wtilde{r}(u_\Lambda) \Vert_{\phi^*}  
= \left(\sum_{k \in \wtilde{\Lambda}} | (\wtilde{r})_{k}^{\, \widehat{}}|^2\right)^{1/2}\; 
\end{equation}
where $(\wtilde{r})_{k}^{\, \widehat{}}$ are the coefficients of $\wtilde{r}(u_\Lambda)$ with respect to the dual basis $\phi^*$.
In order to adaptively increase the accuracy of the approximation of the Galerkin solution, we apply a
D\"orfler-marking (or bulk-chasing) strategy. Precisely, given $\theta\in (0,1)$ we look for a minimal set $\Lambda^*\subset \mathcal{K}$ such that $${\text{Est}}(u_\Lambda;\Lambda^*):=\|P^*_{\Lambda^*} \wtilde{r}(u_\Lambda) \|_{\phi^*} \geq \theta \, {\text{Est}}(u_\Lambda),$$
which is equivalent to $\Vert\wtilde{r}(u_\Lambda)-P^*_{\Lambda^*} \wtilde{r}(u_\Lambda)\Vert_{\phi^*} \leq \sqrt{1-\theta^2}  \Vert \wtilde{r}(u_\Lambda) \Vert_{\phi^*}$. A set $\Lambda^*$ of minimal cardinality can be immediately determined 
rearranging the coefficients $({\tilde r})^{\,\widehat{}}_k$ in non-increasing order of modulus. 
Next, exploiting the exponential decay of the entries of the stiffness matrix inverse (see Proposition \ref{prop:decay} and Property \ref{prop:inverse.matrix-estimate}), we enrich the set $\Lambda^*$ by considering its neighborhood of some radius $J$ depending on $\theta$ and the constants in \eqref{eq:trunc-invmatr-err}. This will guarantee that the convergence of our algorithm can be made arbitrarily fast by choosing $\theta$ sufficiently close to $1$. As a final ingredient,  we introduce a coarsening procedure which removes the negligible components of the Galerkin solution at the expense of a controlled increase of the approximation error. This stage is crucial to guarantee the optimality (in a sense made precise later on) of the approximate solution produced by our adaptive algorithm.
%%%%%%%%%%%%%%%%%%%%%%%%%%%%%%%%%%%%%%%%%%%%%%%%%%%%%%%%%%%%%%%%%%%%%%%%%%%%%%%%%%%%%%
\subsection{FPC-ADLEG: a feasible adaptive algorithm} \label{sec:defADLEG}
%%%%%%%%%%%%%%%%%%%%%%%%%%%%%%%%%%%%%%%%%%%%%%%%%%%%%%%%%%%%%%%%%%%%%%%%%%%%%%%%%%%%%%
We now introduce the following procedures, by which we build our adaptive algorithm.

\begin{itemize}
\item $u_\Lambda := {\bf GAL}(\Lambda)$ \\
Given a finite subset $\Lambda \subset \cal K$, the output
$u_\Lambda \in V_\Lambda$ is the solution of the Galerkin problem (\ref{eq:four.2}) relative to $\Lambda$.

\item $\wtilde r := \text{\bf F-RES}(v_\Lambda,\delta)$ \\
Given $\delta \in (0,1)$ and a function $v_\Lambda \in V_\Lambda$ for some finite index set $\Lambda$,  the module builds an approximate residual  $\wtilde{r}(v)$ such that \eqref{eq:four.2.6} holds. This is accomplished by building suitable finite approximations of the image $\mathcal{L}v_\Lambda$ and of the right-hand side $f$ (see \cite{CNV:mathcomp}, Sect. 3.2 for further details). Employing a feasible residual allows us to work with a finite set of { dual basis
functions $\phi_k^*$} (i.e., with the index $k$ belonging to a finite dimensional subset of ${\cal K}$), or, equivalently, to  { involve in the expression of the residual} only an upper-left (finite) section of the infinite-dimensional matrix $G_t$ and not the whole $G_t$. As already mentioned, we assume that the polynomial degrees of the basis functions activated during the adaptive algorithm (see below the sets $\partial\Lambda_{n}$ produced by the D\"orfler modulus) never exceed a given maximum degree {$p_{\text{max}}$. This, in turn, is equivalent to assuming that the value of  $p_{\text{max}}$, which determines the construction of $G^{p_{\text{max}}}$ via the Gram-Schmidt procedure, is chosen so large that all the generated upper-left (finite) sections $G^{p}_t$ are contained in $G^{p_{\text{max}}}$.} Clearly, such a choice of $p_{\text{max}}$ is related to the value of the parameter $\delta$ and to the tolerance $tol$ employed to stop the algorithm. The possibility of adaptively increasing $p_{\text{max}}$ during the algorithm in order to fulfil this requirement will be considered elsewhere. \\

\item $\Lambda^* := \text{\bf D\"ORFLER}(r, \theta)$\\
Given $\theta \in (0,1)$ and an element $r \in V'$ having a finite expansion, 
the ouput $\Lambda^* \subset \cal K$ is a finite set of minimal cardinality
such that the following inequality holds:
\begin{equation}\label{eq:four.2.5.5}
\Vert P^*_{\Lambda^*} r \Vert_{\phi^*} \geq \theta  \Vert r \Vert_{\phi^*} \;.
\end{equation}
\end{itemize}

\begin{itemize}
\item $\Lambda^* := \text{\bf ENRICH}(\Lambda,J)$ \\
Given an integer $J \geq 0$ and a finite set $\Lambda \subset \mathcal{K}$, the output is the set
$$
\Lambda^* := 
\{ k \in \mathcal{K}\ : \ \text{ there exists } \ell \in \Lambda \text{  such that } \|k - \ell\|_{\ell^1} \leq J \} \;.
$$
Note that since the procedure adds a $2$-dimensional ball of radius $J$ in the $\ell^1$-metric around each point of $\Lambda$, the cardinality
of the new set $\Lambda^*$ can be estimated as $|\Lambda^*| \lesssim 2 J^2 |\Lambda|$.
\item $\Lambda^* := \text{\bf E-D\"ORFLER}(r, \theta)$\\
Given $\theta \in (0,1)$ and an element $r \in H^{-1}(I)$ with finite expansion, 
the ouput $\Lambda^* \subset \mathcal{K}$ is defined by the sequence 
\begin{equation}\label{eq:aggr1}
\widetilde{\Lambda}:= \text{\bf D\"ORFLER}(r, \theta), \qquad \Lambda^{*}:=\text{\bf ENRICH}(\widetilde{\Lambda},J_\theta) \;
\end{equation}
where $J_\theta$ is the smallest integer for which 
$\psi_{{A}^{-1}} (J_\theta,\bar{\gamma})=C_{{A}^{-1}} {\rm e}^{-\bar{\gamma} J_\theta} \leq 
{\frac{\beta^2_*}{d^*}}\sqrt{\frac{1-\theta^2}{\alpha_* \alpha^*}}$ (recall Property \ref{prop:matrix-estimate}).  {\color{red}}

\item $\Lambda := {\bf COARSE}(w, \epsilon)$\\
Given a function $w \in V_{\Lambda^*}$ for some finite index set $\Lambda^*$, and an accuracy $\epsilon>0$
which is known to satisfy $\Vert u -  w \Vert_{H_0^1(\Omega)} \leq  \epsilon$,
 the output $\Lambda \subseteq \Lambda^*$ is a set of minimal cardinality such that
\begin{equation}\label{eq:def-coarse}
\Vert w - P_\Lambda w \Vert_{\phi} \leq 2 \beta_* \epsilon \;.
\end{equation}
\end{itemize}

We are now ready to introduce our adaptive algorithm, which we call Feasible Predictor-Corrector ADaptive LEGendre method ({\bf FPC-ADLEG}). Given a tolerance $tol \in [0,1)$, a marking parameter $\theta \in (0,1)$
 and a feasibility parameter $0<\delta<\sqrt{1-\theta^2}$,  {\bf FPC-ADLEG} reads as follows. 

\medskip 

{\bf Algorithm FPC-ADLEG}($\theta, \delta, \ tol$)
\begin{itemize}
\item[\ ] Set $u_0:=0$, \ $\Lambda_0:=\emptyset$, $n=-1$
\item[\ ]  $\wtilde{r}_{0}:=\text{\bf F-RES}(u_0, \delta)$	
\item[\ ] do
	\begin{itemize}
	\item[\ ] $n \leftarrow n+1$
	\item[\ ] ${\partial\Lambda}_{n}:= 
          \text{\bf E-D\"ORFLER}(\wtilde{r}_{n}, \theta)$		
        \item[\ ] $\widehat\Lambda_{n+1}:=
			\Lambda_{n} \cup {\partial\Lambda}_{n}$
	\item[\ ] $\widehat{u}_{n+1}:= {\bf GAL}(\widehat\Lambda_{n+1})$
	%\item[\ ] $\widehat{r}_{n+1}:= {\bf RES}(\widehat{u}_{n+1})$
	\item[\ ] $\Lambda_{n+1}:=
		{\bf COARSE}\left(\widehat{u}_{n+1}, 
                  {3\frac{{\beta^*}}{\alpha_*} \sqrt{1-\theta^2}\,\Vert \wtilde{r}_n \Vert_{\phi^*}}\right)$
	\item[\ ] $u_{n+1}:={\bf GAL}(\Lambda_{n+1})$
	\item[\ ] $\wtilde{r}_{n+1}:=\text{\bf F-RES}(u_{n+1},\delta)$	
	\end{itemize} 
\item[\ ]  while $\Vert \wtilde{r}_{n+1} \Vert_{\phi^*} > \frac{tol}{1+\delta} $
\end{itemize}

\begin{theorem}[{contraction property of {\bf FPC-ADLEG}}]\label{toe:four2}
Setting $\rho:= 9 \frac{\alpha^*}{\alpha_*}\frac{\beta^*}{\beta_*}\frac{\sqrt{1-\theta^2}}{1-\delta}$ then the errors $u-u_n$ generated for $n \geq 0$ by 
the algorithm satisfy the inequalities
\begin{equation}\label{eq:adgev.2}
\| u-u_{n+1} \|_{H^1_0(\Omega)} \leq \rho \|u-u_{n} \|_{H^1_0(\Omega)}. 
\end{equation}
Therefore, if $\theta$ is chosen in such a way that $\rho<1$, 
for any $tol>0$ the algorithm terminates in a finite number of iterations, whereas for $tol=0$
the sequence $u_n$ converges to $u$ in $H^1_0(\Omega)$ as $n \to \infty$.
\endproof
\end{theorem}
\begin{proof} The contraction factor $\rho$ can be estimated following the same guidelines used in \cite[Sect. 3.3]{CNV:mathcomp} for getting formula (3.23) therein;
obviously, here we use eqns. (\ref{eq:four.2.1bis}) and \eqref{eq:four.2.3bis}. From \eqref{eq:adgev.2} 
it is standard to deduce the bounds 
$\| u-u_{n} \|_{H^1_0(\Omega)} \leq \rho^n \|u-u_{0} \|_{H^1_0(\Omega)}$ for any $n\geq 0$, which imply the convergence of the algorithm for $tol=0$ and the finite termination property for $tol>0$ using the left-hand side estimate in \eqref{eq:four.2.3bis}.
\end{proof}

Note that each iteration of {\bf PC-ADLEG} can be viewed as a predictor step followed by a corrector step. 
The predictor step guarantees an arbitrarily large error reduction (by suitably
enriching the output set from the D\"orfler procedure) at the expense of 
possibly activating an unnecessarily large number of  basis functions. 
The coarsening procedure acts as a corrector step which removes the negligible components of the output of the predictor step. 
The quantitative description of this mechanism  will be possible after we introduce suitable sparsity classes.

%%%%%%%%%%%%%%%%%%%%%%%%%%%%%%%%%%%%%%%%%%%%%%%%%%
\subsection{Nonlinear approximation and sparsity classes}
Given any $v\in V$ we define its {\sl best $N$-term approximation error} as 
$$
E_N(v)= \inf_{\Lambda \subset {\cal K} , 
\ |\Lambda|=N} \Vert v - P_{\Lambda} v \Vert_\phi \;.
$$
%A way to construct a {\sl best $N$-term approximation} $v_N$ of $v$ consists of
%rearranging the coefficients of $v$ in decreasing order of modulus
%$$
%\vert \hat{v}_{k_1}\vert \geq \ldots \geq \vert \hat{v}_{k_n} \vert 
%\geq \vert \hat{v}_{k_{n+1}} \vert \geq \dots 
%$$
%and setting $v_N=P_{\Lambda_N}v$ with $\Lambda_N = \{ k_n \ : \ 1 \leq
%n \leq N \}$. 
%From now on, let us denote by $v_n^*=\hat{v}_{k_n}$
%the rearranged BS coefficients of $v$. Then,
%$$
%E_N(v)= \Vert v - {P_{\Lambda_N} v} \Vert_\phi =  \left(\sum_{n>N} |v_n^*|^2 d_k \right)^{1/2} \;.
%$$
We will be  interested in classifying functions according to the decay law of their best $N$-term approximations, as $N\to\infty$, i.e., according to the ``sparsity'' of their expansions along the NOBS basis. 
In particular, we will consider the following exponential class.
\begin{definition}[{Exponential class of functions}]\label{def:AGev} 
For $\gamma >0$ and $0 < q \leq 2$, we denote by ${\mathcal A}^{\gamma,q}_G$ the {set} defined as
$$
{\mathcal A}^{\gamma,q}_G { := \Big\{ v \in V \ : 
\ \Vert v \Vert_{{\mathcal A}^{\gamma,q}_G}:= 
\sup_{N \geq 0} \, E_N(v) \, \exp({\gamma (N/2)^{q/2} }) < +\infty \Big\} \;.}
$$
\end{definition}
\noindent For functions $v$ in ${\mathcal A}^{\gamma,q}_G$ one can estimate the minimal cardinality of a set $\Lambda$ such that  $\|v-P_\Lambda v\|_\phi \lesssim \varepsilon $ as follows 
\begin{equation}\label{bound:optimal}
\vert \Lambda \vert \leq  \frac{2}{\gamma^{2/q}} \left( \log  \frac{\Vert v \Vert_{{\mathcal A}^{\gamma,q}_G}}{\varepsilon} \right)^{2/t} +1. 
\end{equation}

We note that the class of functions that are analytic in an ellipsoid containing in its interior the set $\bar\Omega$ belongs to ${\mathcal A}^{\gamma,1}_G$. More generally, functions that are not analytic in $\bar \Omega$ but possess a certain Gevrey regularity 
belong to ${\mathcal A}^{\gamma,q}_G $ for some $0<q<1$ (see 
\cite{Canuto-Nochetto-Verani-CMA:2014,BG:72} for more details).

Let us assume that the solution $u$ to \eqref{eq:four03} belongs to some $\mathcal{A}^{\gamma,q}_G$. The optimality of an algorithm for approximating $u$ is defined as the capability of constructing, for any $\epsilon >0$, a finite dimensional approximation $u_\Lambda$ satisfying $\| u- u_\Lambda\|_{H^1_0(\Omega)} \leq \epsilon$ with the cardinality of $\Lambda:=\text{supp} u_\Lambda$ bounded as in \eqref{bound:optimal}, possibly up to some additive constant. A further optimality requirement concerns the cardinality of the supports of all the intermediate functions introduced by the algorithm in order to compute $u_\Lambda$: 
these cardinalities should all be proportional to $\vert \Lambda\vert$, with a proportionality constant independent of $\epsilon$.

For the analysis of the optimality of our algorithm it is important to investigate 
the sparsity class of the image ${\mathcal L}v$ for the operator ${\mathcal L}$ defined in \eqref{eq:four03}, when
the function $v$ belongs to the sparsity class
$\mathcal{A}^{\gamma,q}_G$. The proof is omitted, as it is similar to the one of \cite[Proposition 5.2]{CNV:mathcomp}.
\begin{proposition}[{Continuity of ${\mathcal L}$ in $\mathcal{A}^{\gamma,q}_G$}]\label{propos:spars-res}
Let the differential operator ${\mathcal L}$ be such that the corresponding stiffness matrix satisfies 
${A}_\phi \in {\mathcal D}_e(\gamma_{\mathcal L})$ for some constant $\gamma_{\mathcal L}>0$ (recall Proposition \ref{prop:decay}).
Assume that $v \in {\mathcal A}^{\gamma,q}_G$ for some $\gamma>0$ and $q \in (0,2]$. 
Let one of the two following set of conditions be satisfied.
\begin{enumerate}
\item[\rm (a)]
If the matrix ${A}_\phi$ is banded with $2m+1$ non-zero
diagonals, {let us set}
$$
\bar{\gamma}= \frac{\gamma}{(2m+1)^{q/2}} \;, \qquad \bar{q}= q \;.
$$
\item[\rm (b)]
{If the matrix ${A}_\phi$ is dense}, but the constants $\gamma_{\mathcal L}$ and $\gamma$ satisfy
the inequality $\gamma<  2^{q/2} \gamma_{\mathcal L}$, {let us set} 
$$
\bar{\gamma}= \zeta(q)\gamma \;, \qquad \bar{q}= \frac{q}{1+q} \;,
$$
where we define 
\begin{equation}\label{aux-funct}
\zeta(q) :=  \left( \frac{1+q}{8 \, 2^{q}} \right)^{\frac{q}{2(1+q)}}\;.
\end{equation}
\end{enumerate}
Then, one has ${\cal L}v \in {\mathcal A}^{\bar{\gamma},\bar{q}}_G$, with
\begin{equation}\label{eq:spars11bis}
\Vert {\cal L}v \Vert_{{\mathcal A}_G^{\bar{\gamma},\bar{q}}} \lsim 
\Vert v \Vert_{{\mathcal A}_G^{\gamma,q}} \;.
\end{equation}
\end{proposition}
This result indicates that the residual is expected to belong to a less favorable sparsity class  than the one of the solution.

\smallskip
We are ready to relate the cardinality of the finite expansions activated by the algorithm {\bf FPC-ADLEG} to the sparsity class of the solution. 

\begin{theorem}[{Cardinalities in {\bf FPC-ADLEG}}]\label{teo:pc-ADLEG}
Suppose that $u \in {\mathcal A}^{\gamma,q}_G$, for some $\gamma >0$ and $q \in (0,2]$. 
Then, there exists a constant $C>1$ independent of $\theta$ such that  
\begin{equation}\label{eq:8.3a}
|{\rm supp}\, u_n| \leq \frac{2}{\gamma^{2/q}}
\left( \log  \frac{\Vert u \Vert_{{\mathcal A}^{\gamma,q}_G}}{\ \Vert u-u_n \Vert_{H^1_0(\Omega)}} + \log C \right)^{2/q} +1 \;, 
{\qquad\forall\ n\ge0.}
\end{equation}
If, in addition, the assumptions of Proposition \ref{propos:spars-res} are satisfied, then 
there exists a constant $C_*>1$ proportional to $(1-\theta^2)^{-1/2}$ such that 
 the feasible residual $\wtilde r(u_n)$ satisfies
\begin{equation}\label{eq:8.3b}
\vert {\rm supp}~\wtilde{r}(u_{n}) \vert  \leq \frac{2}{{\wtilde{\gamma}}^{2/\bar{q}} }
\left(\log  \frac{\Vert u \Vert_{{\mathcal A}^{\gamma,q}_G}}{\ \| u-u_{n} \|_{H^1_0(\Omega)}} 
+\log C_*\right)^{2/\bar{q}} +1
\qquad\forall \, n\ge0 \,,
\end{equation}
where $\wtilde{\gamma}= 2^{{-\bar q}/{2}} \bar\gamma$ and $\bar{\gamma}\leq \gamma$, $\bar{q}\leq q$  are the parameters 
defined in Proposition \ref{propos:spars-res}. Moreover, the intermediate Galerkin solution $\hat u_{n+1}$ computed in the predictor step satisfies 
\begin{equation}\label{eq:8.3c}
| {\rm supp}\, \hat u_{n+1}| \leq \frac{2}{\hat{\gamma}^{d/\bar{q}}}
\left( \log  \frac{\Vert u \Vert_{{\mathcal A}^{\gamma,q}_G}}{\ \| u-u_{n} \|_{H^1_0(\Omega)}} 
+\log C_*\right)^{2/\bar{q}} +1 \;, 
{\qquad\forall\ n\ge0\;,}
\end{equation}
where $\hat{\gamma}$ is defined by the relation $\hat{\gamma}^{-2/\bar{q}}=\gamma^{-2/q}+ 2 J_\theta^2 \bar{\gamma}^{-2/\bar{q}}$ with $J_\theta$ introduced in {\bf E-D\"ORFLER}.
\end{theorem}
\begin{proof}
The results can be established following the guidelines used in \cite{CNV:mathcomp}
 for the Fourier case: see, in particular, the proof of Theorem 8.1 for establishing (\ref{eq:8.3a})
and the proof of Theorem 8.3  for establishing (\ref{eq:8.3b})-(\ref{eq:8.3c}). We omit the details since the essential ingredients are the same as in the quoted reference.
\end{proof}

The theorem indicates that the predictor step is
driven by the sparsity class of the residual, which in view of
Proposition \ref{propos:spars-res} may be worse than the sparsity class of the exact solution; 
therefore, optimality with respect to the latter class is not guaranteed.
On the other hand, the corrector step brings  the cardinality close to the optimal one determined by the sparsity class of the solution
(compare (\ref{eq:8.3a}) to (\ref{bound:optimal}) in which $v=u$ and $\varepsilon = \Vert u-u_n \Vert_{H^1_0(\Omega)}$).

\section*{Appendix}
\begin{proof}[Proof of Lemma \ref{lm:L}.]
In the following we will make extensive use of the following property  of the product of univariate Legendre polynomials (see, e.g., \cite{Adams:1878}):
\begin{equation}\label{product}
L_m(x_i)L_n(x_i)=\sum_{r=0}^{\min (m,n)} A_{m,n}^r L_{m+n-2r} (x_i)\qquad i=1,2
\end{equation}
with 
$$ A_{m,n}^r:= \frac{A_{m-r} A_r A_{n-r}}{A_{n+m-r}}\frac{2n+2m-4r+1}{2n+2m-2r+1}$$
and $$ A_0:=1\ , \qquad A_m:=\frac{1\cdot 3 \cdot 5 \ldots (2m-1)}{m !}=\frac{(2m)!}{2^m (m!)^2}.$$

Moreover we recall 
the following asymptotic estimates (see \cite{Canuto-Nochetto-Verani-CMA:2014}):

\begin{enumerate}[$\bullet$]
\item Case $0< r <\min(m,n)$:
\begin{equation}\label{asympt:1}
 \frac{A_{m-r} A_r A_{n-r}}{A_{n+m-r}}\sim 
 \frac{1}{\pi} \frac{\sqrt{n+m-r}}{\sqrt{m-r}\sqrt{n-r}\sqrt{r}}\ ;
\end{equation}
\item Case $r=0$:
\begin{equation}\label{asympt:2}
 \frac{A_{m} A_{n}}{A_{n+m}}\sim \frac{1}{\sqrt{\pi}} \frac{\sqrt{n+m}}{\sqrt{nm}} \ ;
\end{equation}
\item Case $r=\min(m,n)$ and $m\not= n$:
\begin{equation}\label{asympt:3}
\frac{A_{\min(m,n)} A_{\vert m -n \vert}}{{{A_{\max(m,n)}}}}\sim \frac{1}{\sqrt{\pi}} \frac{\sqrt{\max(m,n)}}{\sqrt{\min(m,n)}\sqrt{\vert m -n \vert}}\ . 
\end{equation}
When $m=n$ it is sufficient to use $A_0=1$ to get $\frac{A_{m} A_{0}}{A_{m}}=1$.
\end{enumerate}

We begin from the following expression:
\begin{equation}
a^\eta_{mn}=\int_{\Omega} \nu(x)\nabla\eta_m(x)\nabla\eta_n(x) dx + \int_\Omega \sigma(x)\eta_m(x)\eta_n(x) dx=: a_{mn}^{(1)} + a^{(0)}_{mn}.
\end{equation}
We first estimate $a^{(1)}_{mn}$.  Let $m=(m_1,m_2)$ and $n=(n_1,n_2)$ then using the notation $\nu:=\nu(x_1,x_2)$ and the relation $\eta^\prime_k(x_i)=-\sqrt{k-1/2} L_{k-1}(x_i)$, $i=1,2$, we have
\begin{eqnarray}
a^{(1)}_{mn}&=&\int_\Omega \nu \eta^\prime_{m_1}(x_1)\eta^\prime_{n_1}(x_1)\eta_{m_2}(x_2)
\eta_{m_2}(x_2) dx_1 dx_2 + \int_\Omega \nu \eta_{m_1}(x_1)\eta_{n_1}(x_1)\eta^\prime_{m_2}(x_2)\eta^\prime_{m_2}(x_2) dx_1 dx_2\nonumber\\
&=&B^1_{m,n}\int_{\Omega} \nu L_{m_1-1}(x_1)L_{n_1-1}(x_1)[L_{m_2-2}(x_2)-L_{m_2}(x_2)][L_{n_2-2}(x_2)-L_{n_2}(x_2)]dx_1dx_2\nonumber\\
&&\!\!\!\!\!\!+~B^2_{m,n}\int_{\Omega} \nu [L_{m_1-2}(x_1)-L_{m_1}(x_1)][L_{n_1-2}(x_1)-L_{n_1}(x_2)]  L_{m_2-1}(x_2)L_{n_2-1}(x_2)dx_1dx_2\nonumber\\
&=:& J_1 + J_2
\end{eqnarray}
where we set $$B^1_{m,n}:=\frac{\sqrt{m_1-1/2}\sqrt{n_1-1/2}}{\sqrt{4m_2-2}\sqrt{4n_2-2}}\qquad 
B^2_{m,n}:= \frac{\sqrt{m_2-1/2}\sqrt{n_2-1/2}}{\sqrt{4m_1-2}\sqrt{4n_1-2}}.$$ 
Let us focus on the first term $J_1$. Straightforward calculations yield 
\begin{eqnarray}
J_1&=&B^1_{m,n}\Big\{
\int_{\Omega} \nu L_{m_1-1}(x_1)L_{n_1-1}(x_1) L_{m_2-2}(x_2)L_{n_2-2}(x_2) dx_1dx_2\nonumber\\
&&- \int_{\Omega} \nu L_{m_1-1}(x_1)L_{n_1-1}(x_1) L_{m_2}(x_2)L_{n_2-2}(x_2) dx_1dx_2\nonumber\\
&&- \int_{\Omega} \nu L_{m_1-1}(x_1)L_{n_1-1}(x_1) L_{m_2-2}(x_2)L_{n_2}(x_2) dx_1dx_2\nonumber\\
&&+ \int_{\Omega} \nu L_{m_1-1}(x_1)L_{n_1-1}(x_1) L_{m_2}(x_2)L_{n_2}(x_2) dx_1dx_2\Big\}\nonumber\\
&=:& J_1^1+ J_1^2+ J_1^3+ J_1^4.\nonumber
\end{eqnarray}
For the ease  of presentation we only show how to estimate $J_1^1$ as the other terms can be worked out similarly. Employing  \eqref{product} we obtain 
\begin{eqnarray}
J_1^1&=&B_{m,n}^1 \int_\Omega \nu \sum_{r_1=0}^{\min(m_1-1,n_1-1)} A_{m_1-1,n_1-1}^{r_1} L_{m_1+n_1-2-2r_1}(x_1) \nonumber\\
&&\qquad\sum_{r_2=0}^{\min(m_2-2,n_2-2)} A_{m_2-2,n_2-2}^{r_2} L_{m_2+n_2-4-2r_2}(x_2) ~dx_1x_2 
\nonumber\\
&=&B_{m,n}^1  \sum_{r_1=0}^{\min(m_1-1,n_1-1)} \sum_{r_2=0}^{\min(m_2-2,n_2-2)} 
A_{m_1-1,n_1-1}^{r_1} A_{m_2-2,n_2-2}^{r_2} 
\int_\Omega \nu L_{m_1+n_1-2-2r_1}L_{m_2+n_2-4-2r_2}dx_1x_2. \nonumber
\end{eqnarray}
Using  the multidimensional Legendre expansion $\nu(x)=\sum_{k\in{\cal K}} \nu_k L_k(x)$
we obtain
\begin{eqnarray}
J_1^1&=&4 B_{m,n}^1 \sum_{r_1=0}^{\min(m_1-1,n_1-1)} \sum_{r_2=0}^{\min(m_2-2,n_2-2)} 
\frac{A_{m_1-1,n_1-1}^{r_1} A_{m_2-2,n_2-2}^{r_2} \nu_{m_1+n_1-2-2r_1,m_2+n_2-4-2r_2}}{[2(m_1+n_1-2-2r_1)+1][2(m_2+n_2-4-2r_2)+1]}.\nonumber
\end{eqnarray}
We now employ the asymptotic estimates  \eqref{asympt:1}-\eqref{asympt:3}  to bound the terms $A_{m_1-2,n_1-2}^{r_1}$ and $A_{m_2-2,n_2-2}^{r_2}$. Accordingly, we need to distinguish among several cases depending on the combination of the values assumed by $r_1$ and $r_2$. However, for the ease of reading, we only consider the case $0<r_1<\min(m_1-2,n_1-2)$ and $0<r_2<\min(m_2-2,n_2-2)$ as the other ones can be treated similarly. In this case, \eqref{asympt:1}  yields 
\begin{equation}
A_{m_1-2,n_1-2}^{r_1}\simeq \frac{1}{\pi} \frac{\sqrt{n_1+m_1-4-r_1}}{\sqrt{m_1-2-r_1}\sqrt{n_1-2-r_1}\sqrt{r_1}} \frac{2(n_1-2)+2(m_1-2)-4r_1+1}{2(n_1-2)+2(m_1-2)-2r_1+1}.
\end{equation}
A similar estimate holds also for  $A_{m_2-2,n_2-2}^{r_2}$. Thus we have 
\begin{eqnarray}
 &&\frac{ B_{m,n}^1 A_{m_1-1,n_1-1}^{r_1} A_{m_2-2,n_2-2}^{r_2}} {[2(m_1+n_1-2-2r_1)+1][2(m_2+n_2-4-2r_2)+1]} \simeq \nonumber\\
&&\simeq \frac{\sqrt{2m_1-1}\sqrt{2n_1-1}}{\sqrt{2m_2-1}\sqrt{2n_2-1}}
\frac{\sqrt{n_2+m_2-4-r_2}\sqrt{n_1+m_1-4-r_1} }{\sqrt{m_2-2-r_2}\sqrt{n_2-2-2r_2}\sqrt{r_2}
\sqrt{m_2-2-r_2}\sqrt{n_2-2-2r_2}\sqrt{r_2}}\nonumber\\ 
&&\qquad\frac{1}{(2m_2+2n_2-2r_2-7)(2m_1+2n_1-2r_1-3)}
\nonumber\\
&&\simeq\frac{\sqrt{m_1n_1}}{\sqrt{m_1-1-r_1}\sqrt{n_1-1}\sqrt{r_1}}
\frac{1}{\sqrt{n_1+m_1-2-r_1}}
\frac{1}{\sqrt{n_2+m_2-4-r_2}} 
\frac{1}{\sqrt{2m_2-1}\sqrt{2n_2-1}}\nonumber\\
&&\qquad\frac{1}{\sqrt{m_2-2-r_2}\sqrt{n_2-2-2r_2}\sqrt{r_2}}\nonumber\\
&&\simeq
\frac{1}{\sqrt{\min(m_1+n_1-2,\vert m_1-n_1\vert)}}
\frac{1}{\sqrt{n_2+m_2-4-r_2}} 
\frac{1}{\sqrt{2m_2-1}\sqrt{2n_2-1}}\nonumber\\
&&\qquad\frac{1}{\sqrt{m_2-2-r_2}\sqrt{n_2-2-2r_2}\sqrt{r_2}}\nonumber\\
&&\lesssim 1.\nonumber
\end{eqnarray}
Thus we have
\begin{equation}
J_1^1 \lesssim \sum_{r_1=0}^{\min(m_1-1,n_1-1)} \sum_{r_2=0}^{\min(m_2-2,n_2-2)} 
 \nu_{m_1+n_1-2-2r_1,m_2+n_2-4-2r_2}.
 \end{equation}
 Similar estimates can be obtained for the terms $J_1^2,\ldots,J_1^4$ yielding
 \begin{eqnarray}
 J_1&\lesssim&  
 \sum_{r_1=0}^{\min(m_1-1,n_1-1)} \sum_{r_2=0}^{\min(m_2-2,n_2-2)} 
 \nu_{m_1+n_1-2-2r_1,m_2+n_2-4-2r_2}\nonumber\\
 &&+
 \sum_{r_1=0}^{\min(m_1-1,n_1-1)} \sum_{r_2=0}^{\min(m_2,n_2-2)} 
 \nu_{m_1+n_1-2-2r_1,m_2+n_2-2-2r_2}\nonumber\\
 &&+
  \sum_{r_1=0}^{\min(m_1-1,n_1-1)} \sum_{r_2=0}^{\min(m_2-2,n_2)} 
 \nu_{m_1+n_1-2-2r_1,m_2+n_2-2-2r_2}\nonumber\\
 &&+
  \sum_{r_1=0}^{\min(m_1-1,n_1-1)} \sum_{r_2=0}^{\min(m_2,n_2)} 
 \nu_{m_1+n_1-2-2r_1,m_2+n_2-2r_2}.\nonumber
  \end{eqnarray}
  Anaologously, we can prove the following estimate for the term $J_2$
   \begin{eqnarray}
 J_2&\lesssim&  
 \sum_{r_1=0}^{\min(m_1-2,n_1-2)} \sum_{r_2=0}^{\min(m_2-1,n_2-1)} 
 \nu_{m_1+n_1-4-2r_1,m_2+n_2-2-2r_2}\nonumber\\
 &&+
 \sum_{r_1=0}^{\min(m_1,n_1-2)} \sum_{r_2=0}^{\min(m_2-1,n_2-1)} 
 \nu_{m_1+n_1-2-2r_1,m_2+n_2-2-2r_2}\nonumber\\
 &&+
  \sum_{r_1=0}^{\min(m_1-2,n_1)} \sum_{r_2=0}^{\min(m_2-1,n_2-1)} 
 \nu_{m_1+n_1-2-2r_1,m_2+n_2-2-2r_2}\nonumber\\
 &&+
  \sum_{r_1=0}^{\min(m_1,n_1)} \sum_{r_2=0}^{\min(m_2-1,n_2-1)} 
 \nu_{m_1+n_1-2r_1,m_2+n_2-2-2r_2}.\nonumber
  \end{eqnarray}
Assuming $ \vert \nu_k \vert  \leq C_\eta e^{-\gamma \|k\|_{\ell^{1}}}$ for every $k\in{\cal K}$ and employing the above estimates for $J_1$ and $J_2$, we obtain
\begin{eqnarray}
\vert a_{mn}^{(1)}\vert 
 &\lesssim&   e^{-\gamma (\vert m_1-n_1\vert + \vert m_2-n_2\vert)}
\Big\{ \sum_{r_1=0}^{\min(m_1-1,n_1-1)}  \sum_{r_2=0}^{\min(m_2-2,n_2-2)}
e^{-2\gamma (\min(m_1-1,n_1-1)-r_1)} e^{-2\gamma (\min(m_2-2,n_2-2)-r_2)}\nonumber\\
&&+\ldots+ \sum_{r_1=0}^{\min(m_1,n_1)}  \sum_{r_2=0}^{\min(m_2-1,n_2-1)}
e^{-2\gamma (\min(m_1,n_1)-r_1)} e^{-2\gamma (\min(m_2-1,n_2-1)-r_2)}\Big\}\nonumber\\
&\lesssim& e^{-\gamma (\vert m_1-n_1\vert + \vert m_2-n_2\vert)}= C e^{-\gamma \|m-n\|_{\ell^{1}}}.\nonumber
\end{eqnarray}

We now estimate $a^{(0)}_{mn}$. Let $m=(m_1,m_2)$ and $n=(n_1,n_2)$ then recalling \eqref{eq:defBS} and using the notation $\sigma:=\sigma(x_1,x_2)$ we have 
\begin{eqnarray}
a^{(0)}_{mn}&=&\int_\Omega \sigma\eta_{m_1}(x_1)\eta_{n_1}(x_1)\eta_{m_2}(x_2)\eta_{m_2}(x_2)dx_1 dx_2\nonumber\\
&=&
C_m^n
\int_\Omega \sigma
[(L_{m_1-2}-L_{m_1}) (L_{n_1-2}-L_{n_1})](x_1)[(L_{m_2-2}-L_{m_2})(L_{n_2-2}-L_{n_2})](x_2)   
dx_1 dx_2\nonumber\\
&=&C_m^n
\int_\Omega \sigma
[L_{m_1-2}L_{n_1-2}-L_{m_1-2}L_{n_1}-L_{m_1}L_{n_1-2}+L_{n_1}L_{m_1}](x_1)[L_{m_2-2}L_{n_2-2} + \nonumber\\
&&\qquad-L_{m_2-2}L_{n_2}-L_{m_2}L_{n_2-2}+L_{n_2}L_{m_2}](x_2)  dx_1 dx_2\nonumber
\end{eqnarray}
with $C_m^n:=\frac{1}{\sqrt{(4m_1-1)(4m_2-1)(4n_1-1)(4n_2-1)}}$. Now employing \eqref{product} we obtain 
\begin{eqnarray}
a^{(0)}_{mn}&=&C_m^n \int_\Omega \sigma 
\Big[
\sum_{r_1=0}^{\min(m_1-2,n_1-2)} A_{m_1-2,n_1-2}^{r_1} L_{m_1+n_1-4-2r_1}
- \sum_{r_1=0}^{\min(m_1-2,n_1)} A_{m_1-2,n_1}^{r_1} L_{m_1+n_1-2-2r_1}\nonumber\\
&&-\sum_{r_1=0}^{\min(m_1,n_1-2)} A_{m_1,n_1-2}^{r_1} L_{m_1+n_1-2-2r_1}
+\sum_{r_1=0}^{\min(m_1,n_1)} A_{m_1,n_1}^{r_1} L_{m_1+n_1-2r_1}
\Big](x_1)\nonumber\\
&&\Big[
\sum_{r_2=0}^{\min(m_2-2,n_2-2)} A_{m_2-2,n_2-2}^{r_2} L_{m_2+n_2-4-2r_2}
- \sum_{r_2=0}^{\min(m_2-2,n_2)} A_{m_2-2,n_2}^{r_2} L_{m_2+n_2-2-2r_2}\nonumber\\
&&-\sum_{r_2=0}^{\min(m_2,n_2-2)} A_{m_2,n_2-2}^{r_2} L_{m_2+n_2-2-2r_2}
+\sum_{r_2=0}^{\min(m_2,n_2)} A_{m_2,n_2}^{r_2} L_{m_2+n_2-2r_2}
\Big](x_2) dx_1dx_2.\nonumber\\
&=&I_1+\ldots+I_{16}.\nonumber
\end{eqnarray}
We now need to estimate $I_1,\ldots,I_{16}$. To simplify the exposition, we only show how to estimate $I_1$, as the other terms can be treated similarly. 

Using  the multidimensional Legendre expansion $\sigma(x)=\sum_{k\in{\cal K}} \sigma_k L_k(x)$ together  with \eqref{eq:Leg-ort} we get 
\begin{eqnarray}
I_1&=&C_m^n \sum_{r_1=0}^{\min(m_1-2,n_1-2)} 
\sum_{r_2=0}^{\min(m_2-2,n_2-2)}
A_{m_1-2,n_1-2}^{r_1} A_{m_2-2,n_2-2}^{r_2} \int_\Omega \sigma L_{m_1+n_1-4-2r_1}L_{m_2+n_2-4-2r_2}
dx_1dx_2\nonumber\\
&=&
C_m^n  \sum_{r_1=0}^{\min(m_1-2,n_1-2)} 
\sum_{r_2=0}^{\min(m_2-2,n_2-2)}\frac{A_{m_1-2,n_1-2}^{r_1} A_{m_2-2,n_2-2}^{r_2}  \sigma_{m_1+n_1-4-2r_1,m_2+n_2-4-2r_2} 
}{[2(m_1+n_1-4-2r_1)+1][2(m_2+n_2-4-2r_2)+1]}.\nonumber
\end{eqnarray}
We now employ the asymptotic estimates  \eqref{asympt:1}-\eqref{asympt:3}  to bound the terms $A_{m_1-2,n_1-2}^{r_1}$ and $A_{m_2-2,n_2-2}^{r_2}$. Accordingly, we need to distinguish among several cases depending on the combination of the values assumed by $r_1$ and $r_2$. However, for the ease of reading, we only consider the case $0<r_1<\min(m_1-2,n_1-2)$ and $0<r_2<\min(m_2-2,n_2-2)$ as the other ones can be treated similarly. In this case, \eqref{asympt:1}  yields 
\begin{equation}
A_{m_1-2,n_1-2}^{r_1}\simeq \frac{1}{\pi} \frac{\sqrt{n_1+m_1-4-r_1}}{\sqrt{m_1-2-r_1}\sqrt{n_1-2-r_1}\sqrt{r_1}} \frac{2(n_1-2)+2(m_1-2)-4r_1+1}{2(n_1-2)+2(m_1-2)-2r_1+1}.
\end{equation}
A similar estimate holds also for  $A_{m_2-2,n_2-2}^{r_2}$. Hence, we have 
\begin{eqnarray}
\frac{C_m^n A_{m_1-2,n_1-2}^{r_1} A_{m_2-2,n_2-2}^{r_2}}{[2(m_1+n_1-4-2r_1)+1][2(m_2+n_2-4-2r_2)+1]}\lesssim 1.
\end{eqnarray}
Employing \eqref{asympt:2} and  \eqref{asympt:3} yields similar estimates for the cases  $r_1=0,\min(m_1-2,n_1-2)$ and $r_2=0,\min(m_2-2,n_2-2)$. 
In conclusion, we get 
\begin{equation}
I_1\lesssim  \sum_{r_1=0}^{\min(m_1-2,n_1-2)} 
\sum_{r_2=0}^{\min(m_2-2,n_2-2)}\sigma_{m_1+n_1-4-2r_1,m_2+n_2-4-2r_2}. 
\end{equation}
Similar estimates can be obtained for $I_2,\ldots,I_{16}$ thus yielding
\begin{eqnarray}
 a_{mn}^{(0)} &\lesssim&   \sum_{r_1=0}^{\min(m_1-2,n_1-2)}  \sum_{r_2=0}^{\min(m_2-2,n_2-2)}
 \sigma_{m_1+n_1-4-2r_1,m_2+n_2-4-2r_2} \nonumber\\
&& + \ldots +  \sum_{r_1=0}^{\min(m_1,n_1)}  \sum_{r_2=0}^{\min(m_2,n_2)}\sigma_{m_1+n_1-2r_1,m_2+n_2-2r_2}. \nonumber
\end{eqnarray}
Assuming $ \vert \sigma_k \vert  \leq C_\eta e^{-\gamma \|k\|_{\ell^{1}}}$ for every $k\in{\cal K}$ and employing the above estimate, we obtain 
\begin{eqnarray}
\vert a_{mn}^{(0)}\vert  &\lesssim&   e^{-\gamma (\vert m_1-n_1\vert + \vert m_2-n_2\vert)}
\Big\{ \sum_{r_1=0}^{\min(m_1-2,n_1-2)}  \sum_{r_2=0}^{\min(m_2-2,n_2-2)}
e^{-2\gamma (\min(m_1-2,n_1-2)-r_1)} e^{-2\gamma (\min(m_2-2,n_2-2)-r_2)}
\nonumber\\
&& + \ldots +  \sum_{r_1=0}^{\min(m_1,n_1)}  \sum_{r_2=0}^{\min(m_2,n_2)}
e^{-2\gamma (\min(m_1,n_1)-r_1)} e^{-2\gamma (\min(m_2,n_2)-r_2)}\Big\}\nonumber\\
&\lesssim& e^{-\gamma (\vert m_1-n_1\vert + \vert m_2-n_2\vert)}= C e^{-\gamma \|m-n\|_{\ell^{1}}}.\nonumber
\end{eqnarray}
This concludes the proof.
\end{proof}

\section*{Acknowledgements}
The authors would like to thank Michele Benzi for
helpful discussions and for pointing to \cite{Strohmer-et-al:2013}.
The first and  third author have been partially supported by the Italian research grant  {\sl Prin 2012} 2012HBLYE4\_004   ``Metodologie innovative nella modellistica differenziale numerica''.

\def\ocirc#1{\ifmmode\setbox0=\hbox{$#1$}\dimen0=\ht0 \advance\dimen0
  by1pt\rlap{\hbox to\wd0{\hss\raise\dimen0
  \hbox{\hskip.2em$\scriptscriptstyle\circ$}\hss}}#1\else {\accent"17 #1}\fi}


\begin{thebibliography}{}

\bibitem{Adams:1878}
J.~Adams.
\newblock On the expression of the product of any two legendreÕs coefficients
  by means of a series of legendreÕs coefficients.
\newblock {\em Proceedings of the Royal Society of London}, 27:63--71, 1878.

\bibitem{BG:72}
M.~S. Baouendi and C.~Goulaouic.
\newblock R\'egularit\'e analytique et it\'er\'es d'op\'erateurs elliptiques
  d\'eg\'en\'er\'es; applications.
\newblock {\em J. Functional Analysis}, 9:208--248, 1972.

\bibitem{Benzi-Tuma:2000}
M.~Benzi and M.~T{\ocirc{u}}ma.
\newblock Orderings for factorized sparse approximate inverse preconditioners.
\newblock {\em SIAM J. Sci. Comput.}, 21(5):1851--1868 (electronic), 2000.
\newblock Iterative methods for solving systems of algebraic equations (Copper
  Mountain, CO, 1998).

\bibitem{B:tree}
P.~Binev.
\newblock Tree approximation for $hp$-adaptivity.
\newblock in preparation.

\bibitem{B:oberw}
P.~Binev.
\newblock Instance optimality for $hp$-type approximation.
\newblock Technical Report~39, Mathematisches Forschungsinstitut Oberwolfach,
  2013.

\bibitem{BDD:04}
P.~Binev, W.~Dahmen, and R.~DeVore.
\newblock Adaptive finite element methods with convergence rates.
\newblock {\em Numer. Math.}, 97(2):219--268, 2004.

\bibitem{Burg-Doerfler:11}
M.~B{\"u}rg and W.~D{\"o}rfler.
\newblock Convergence of an adaptive {$hp$} finite element strategy in higher
  space-dimensions.
\newblock {\em Appl. Numer. Math.}, 61(11):1132--1146, 2011.

\bibitem{CNSV:14}
C.~Canuto, R.~Nochetto, R.~Stevenson, and M.~Verani.
\newblock An $hp$ adaptive finite element method: convergence and optimality
  properties.
\newblock in preparation.

\bibitem{Canuto-Nochetto-Verani-CMA:2014}
C.~Canuto, R.~Nochetto, and M.~Verani.
\newblock Contraction and optimality properties of adaptive {L}egendre-{G}alerkin
  methods: the 1-dimensional case.
\newblock {\em Computers and Mathematics with Applications}, 67(4):752--770,
  2014.

\bibitem{CNV:mathcomp}
C.~Canuto, R.~H. Nochetto, and M.~Verani.
\newblock {A}daptive {F}ourier-{G}alerkin {M}ethods.
\newblock {\em Math. Comp.}, 83:1645--1687, 2014.

\bibitem{CSV:13}
C.~Canuto, V.~Simoncini, and M.~Verani.
\newblock {O}n the decay of the inverse of matrices that are sum of {K}ronecker
  products.
\newblock {\em Linear Algebra and Its Applications}, 452:21--39, 2014.

\bibitem{CV:book}
C.~Canuto and M.~Verani.
\newblock On the numerical analysis of adaptive spectral/hp methods for
  elliptic problems.
\newblock In {\em Analysis and Numerics of Partial Differential Equations, F.
  Brezzi et al (Eds.)}, pages 165--192. Springer INdAM series, 2013.

\bibitem{Nochetto-et-al:2008}
J.~M. Cascon, C.~Kreuzer, R.~H. Nochetto, and K.~G. Siebert.
\newblock Quasi-optimal convergence rate for an adaptive finite element method.
\newblock {\em SIAM J. Numer. Anal.}, 46(5):2524--2550, 2008.

\bibitem{CDDV:1998}
A.~Cohen, W.~Dahmen, and R.~DeVore.
\newblock Adaptive wavelet methods for elliptic operator equations --
  convergence rates.
\newblock {\em Math. Comp}, 70:27--75, 1998.

\bibitem{CDN:11}
A.~Cohen, R.~DeVore, and R.~H. Nochetto.
\newblock Convergence rates of {AFEM} with {$H^{-1}$} data.
\newblock {\em Found. Comput. Math.}, 12(5):671--718, 2012.

\bibitem{dorfler:96}
W.~D{\"o}rfler.
\newblock A convergent adaptive algorithm for {P}oisson's equation.
\newblock {\em SIAM J. Numer. Anal.}, 33(3):1106--1124, 1996.

\bibitem{Doerfler-Heuveline:07}
W.~D{\"o}rfler and V.~Heuveline.
\newblock Convergence of an adaptive {$hp$} finite element strategy in one
  space dimension.
\newblock {\em Appl. Numer. Math.}, 57(10):1108--1124, 2007.

\bibitem{Duff-Erisman-Reid:89}
I.~S. Duff, A.~M. Erisman, and J.~K. Reid.
\newblock {\em Direct Methods for Sparse Matrices}.
\newblock Clarendon Press, Oxford, 1989.

\bibitem{Golub-VanLoan:1996}
G.~H. Golub and C.~F. Van~Loan.
\newblock {\em Matrix computations}.
\newblock Johns Hopkins Studies in the Mathematical Sciences. Johns Hopkins
  University Press, Baltimore, MD, third edition, 1996.

\bibitem{Babuska-Gui}
W.~Gui and I.~Babu{\v{s}}ka.
\newblock The {$h,\;p$} and {$h$}-{$p$} versions of the finite element method
  in {$1$} dimension. {III}. {T}he adaptive {$h$}-{$p$} version.
\newblock {\em Numer. Math.}, 49(6):659--683, 1986.

\bibitem{Hall-Jin:2010}
P.~Hall and J.~Jin.
\newblock Innovated higher criticism for detecting sparse signals in correlated
  noise.
\newblock {\em Ann. Statist.}, 38(3):1686--1732, 2010.

\bibitem{Jaffard:1990}
S.~Jaffard.
\newblock Propri\'et\'es des matrices "bien localis\'ees" pr\`es de leur
  diagonale et quelques applications.
\newblock {\em Annales de l'I.H.P.}, 5:461--476, 1990.

\bibitem{Strohmer-et-al:2013}
I.~Krishtal, T.~Strohmer, and T.~Wertz.
\newblock Localization of matrix factorizations.
\newblock {\em arXiv:1305.1618}, 2013.

\bibitem{Maitre:1996}
J.-F. Maitre and O.~Pourquier.
\newblock Condition number and diagonal preconditioning: comparison of the
  {$p$}-version and the spectral element methods.
\newblock {\em Numer. Math.}, 74(1):69--84, 1996.

\bibitem{Mitchell:2011}
W.~F. Mitchell and M.~A. McClain.
\newblock A survey of {$hp$}-adaptive strategies for elliptic partial
  differential equations.
\newblock In {\em Recent advances in computational and applied mathematics},
  pages 227--258. Springer, Dordrecht, 2011.

\bibitem{MNS:00}
P.~Morin, R.~H. Nochetto, and K.~G. Siebert.
\newblock Data oscillation and convergence of adaptive {FEM}.
\newblock {\em SIAM J. Numer. Anal.}, 38(2):466--488 (electronic), 2000.

\bibitem{NSV:09}
R.~H. Nochetto, K.~G. Siebert, and A.~Veeser.
\newblock Theory of adaptive finite element methods: an introduction.
\newblock In {\em Multiscale, nonlinear and adaptive approximation}, pages
  409--542. Springer, Berlin, 2009.

\bibitem{Schmidt-Siebert:00}
A.~Schmidt and K.~G. Siebert.
\newblock A posteriori estimators for the {$h$}-{$p$} version of the finite
  element method in 1{D}.
\newblock {\em Appl. Numer. Math.}, 35(1):43--66, 2000.

\bibitem{Stevenson:2007}
R.~Stevenson.
\newblock Optimality of a standard adaptive finite element method.
\newblock {\em Found. Comput. Math.}, 7(2):245--269, 2007.

\bibitem{Stevenson:09}
R.~Stevenson.
\newblock Adaptive wavelet methods for solving operator equations: an overview.
\newblock In {\em Multiscale, nonlinear and adaptive approximation}, pages
  543--597. Springer, Berlin, 2009.


\end{thebibliography}
\end{document}